\documentclass[a4paper11pt]{article}
\usepackage{amsmath, amssymb, amsthm, latexsym}

\newcommand{\N}{\mathbb N}
\newcommand{\Q}{\mathbb Q}

\newcommand{\cantor}{\{0,1\}^\N}
\newcommand{\words}{\{0,1\}^*}

\newcommand{\segment}{\!\upharpoonright \!}
\theoremstyle{plain}
\newtheorem{theorem}{Theorem}
\newtheorem{proposition}[theorem]{Proposition}
\newtheorem{corollary}[theorem]{Corollary}

\newtheorem{definition}[theorem]{Definition}

\newtheorem*{theorem*}{Theorem}
\newtheorem*{corollary*}{Corollary}
\newtheorem*{lemma*}{Lemma}
\newtheorem*{claim*}{Claim}
\newtheorem{example}[theorem]{Example}
\theoremstyle{definition}

\theoremstyle{remark}
\newtheorem{remark}[theorem]{Remark}
\newtheorem*{note*}{Note}
\newtheorem*{proposition*}{Proposition}
\newtheorem*{definition*}{Definition}
\newtheorem*{convention*}{Convention}
\newtheorem*{notation*}{Notation}
\newtheorem*{remark*}{Remark}
\newtheorem*{example*}{Example}
\title{Kolmogorov complexity in perspective}
\author{Marie Ferbus-Zanda\\
{\footnotesize LIAFA, CNRS \& Universit\'e Paris 7}\\
{\footnotesize 2, pl. Jussieu}\\
{\footnotesize 75251 Paris Cedex 05 France}\\
{\footnotesize ferbus@logique.jussieu.fr}
\and Serge Grigorieff\\
{\footnotesize LIAFA, CNRS \& Universit\'e Paris 7}\\
{\footnotesize 2, pl. Jussieu}\\
{\footnotesize 75251 Paris Cedex 05 France}\\
{\footnotesize seg@liafa.jussieu.fr}}
\begin{document}
\maketitle
{\footnotesize\tableofcontents}
\normalsize
%
%%%%%%%%%%%%%%%%%%%%%%%%%%%%%%%%%%%%%%%%%%%%%%%%%%%%%%%%%%%%%%%%%%%
%%%%%%%%%%%%%%%%%%%%%%%%%%%%%%%%%%%%%%%%%%%%%%%%%%%%%%%%%%%%%%%%%%%
%%%%%%%%%%%%%%%%%%%%%%%%%%%%%%%%%%%%%%%%%%%%%%%%%%%%%%%%%%%%%%%%%%%
%%%%%%%%%%%%%%%%%%%%%%%%%%%%%%%%%%%%%%%%%%%%%%%%%%%%%%%%%%%%%%%%%%%
%%%%%%%%%%%%%%%%%%%%%%%%%%%%%%%%%%%%%%%%%%%%%%%%%%%%%%%%%%%%%%%%%%%
%%%%%%%%%%%%%%%%%%%%%%%%%%%%%%%%%%%%%%%%%%%%%%%%%%%%%%%%%%%%%%%%%%%
%%%%%%%%%%%%%%%%%%%%%%%%%%%%%%%%%%%%%%%%%%%%%%%%%%%%%%%%%%%%%%%%%%%
\begin{abstract}
We survey the diverse approaches to the notion of information
content: from Shannon entropy to Kolmogorov complexity.
The main applications of Kolmogorov complexity are presented:
namely, the mathematical notion of randomness
(which goes back to the 60's with the work of Martin-L\"of,
Schnorr, Chaitin, Levin),
and classification, which is a recent idea with provocative
implementation by Vitanyi and Cilibrasi.
.
\end{abstract}
%%%%%%%%%%%%%%%%%%%%%%%%%%%%%%%%%%%%%%%%%%%%%%%%%%%%%%%%%%%%%%%%%%%
%%%%%%%%%%%%%%%%%%%%%%%%%%%%%%%%%%%%%%%%%%%%%%%%%%%%%%%%%%%%%%%%%%%
%%%%%%%%%%%%%%%%%%%%%%%%%%%%%%%%%%%%%%%%%%%%%%%%%%%%%%%%%%%%%%%%%%%
%%%%%%%%%%%%%%%%%%%%%%%%%%%%%%%%%%%%%%%%%%%%%%%%%%%%%%%%%%%%%%%%%%%
%%%%%%%%%%%%%%%%%%%%%%%%%%%%%%%%%%%%%%%%%%%%%%%%%%%%%%%%%%%%%%%%%%%
%%%%%%%%%%%%%%%%%%%%%%%%%%%%%%%%%%%%%%%%%%%%%%%%%%%%%%%%%%%%%%%%%%%
%
\begin{note*}
Following Robert Soare's recommendations in \cite{soareBSL},
which have now gained large agreement, we shall write
{\em computable} and {\em computably enumerable}
in place of the old fashioned
{\em recursive} and {\em recursively enumerable}.
\end{note*}
\begin{notation*}
By $\log(x)$ we mean the logarithm of $x$ in base $2$.
By $\lfloor x\rfloor$ we mean the ``floor" of $x$,
i.e. the largest integer $\leq x$.
Similarly, $\lceil x\rceil$ denotes the ``ceil" of $x$,
i.e. the smallest integer $\geq x$.
Recall that the length of the binary representation of
a non negative integer $n$ is $1+\lfloor\log n\rfloor$.
%%
%\medskip\\
%2. The number of elements of a finite family $\+F$ is denoted
%by $\sharp\+F$.
\end{notation*}
%
%%%%%%%%%%%%%%%%%%%%%%%%%%%%%%%%%%%%%%%%%%%%%%%%%%%%%%%%%%%%%%%%%%%
%%%%%%%%%%%%%%%%%%%%%%%%%%%%%%%%%%%%%%%%%%%%%%%%%%%%%%%%%%%%%%%%%%%
%%%%%%%%%%%%%%%%%%%%%%%%%%%%%%%%%%%%%%%%%%%%%%%%%%%%%%%%%%%%%%%%%%%
\section{Three approaches to the quantitative definition
of information}
%%%%%%%%%%%%%%%%%%%%%%%%%%%%%%%%%%%%%%%%%%%%%%%%%%%%%%%%%%%%%%%%%%%
%%%%%%%%%%%%%%%%%%%%%%%%%%%%%%%%%%%%%%%%%%%%%%%%%%%%%%%%%%%%%%%%%%%
%%%%%%%%%%%%%%%%%%%%%%%%%%%%%%%%%%%%%%%%%%%%%%%%%%%%%%%%%%%%%%%%%%%
%
A title borrowed from Kolmogorov's seminal paper,
1965 \cite{kolmo65}.
%
%
%%%%%%%%%%%%%%%%%%%%%%%%%%%%%%%%%%%%%%%%%%%%%%%%%%%%%%%%%%%%%%%%%%%
\subsection{Which information ?}
%%%%%%%%%%%%%%%%%%%%%%%%%%%%%%%%%%%%%%%%%%%%%%%%%%%%%%%%%%%%%%%%%%%
%
%------------------------------------------------------------------
\subsubsection{About anything...}
%------------------------------------------------------------------
About anything can be seen as conveying information.
As usual in mathematical modelization,
we retain only a few features of some real entity or process,
and associate to them some finite or infinite mathematical objects.
For instance,
\begin{itemize}
\item
- an integer or a rational number or a word in some alphabet,
\\- a finite sequence or a finite set of such objects,
\\- a finite graph,...
\item
- a real,
\\- a finite or infinite sequence of reals or a set of reals,
\\- a function over words or numbers,...
\end{itemize}
This is very much as with probability spaces. For instance,
to modelize the distributions of $6$ balls into $3$ cells,
(cf. Feller's book \cite{feller} \S I2, II5)
we forget everything about the nature of balls and cells
and of the distribution process,
retaining only two questions: ``how many are they?" and
``are they distinguishable or not?".
Accordingly, the modelization considers
\\
- either the $729=3^6$ maps from the set of balls
into the set of cells
in case the balls are distinguishable and so are the cells
(this is what is done in Maxwell-Boltzman statistics),
\\
- or the $28=\binom{3+6-1}{6}$ triples of non negative integers with sum $6$
in case the cells are distinguishable but not the balls
(this is what is done in in Bose-Einstein statistics)
\\
- or the $7$ sets of at most $3$ integers with sum $6$
in case the balls are undistinguishable and so are the cells.
%
%
%------------------------------------------------------------------
\subsubsection{Especially words}\label{sss:words}
%------------------------------------------------------------------
In information theory, special emphasis is made on information
conveyed by words on finite alphabets.
I.e. on {\em sequential information} as opposed to the obviously
massively parallel and interactive distribution of information
in real entities and processes.
A drastic reduction which allows for mathematical developments
(but also illustrates the Italian sentence
``traduttore, traditore!").
\medskip\\
As is largely popularized by computer science, any finite alphabet
with more than two letters can be reduced to one with exactly
two letters.
For instance, as exemplified by the ASCII code
(American Standard Code for Information Interchange),
any symbol used in written English
-- namely the lowercase and uppercase letters, the decimal digits,
the diverse punctuation marks, the space, apostrophe, quote,
left and right parentheses --
can be coded by length $7$ binary words
(corresponding to the $128$ ASCII codes).
Which leads to a simple way to code any English text by
a binary word (which is $7$ times longer).\footnote{For
other European languages which have a lot of
diacritic marks, one has to consider the $256$ codes of the
Extended ASCII code.}
\medskip\\
Though quite rough, the length of a word is the basic measure
of its information content.
Now a fairness issue faces us:
richer the alphabet, shorter the word.
Considering groups of $k$ successive letters as new letters
of a super-alphabet, one trivially divides the length by $k$.
For instance, a length $n$ binary word becomes a length
$\lceil \frac{n}{256}\rceil$ word with the usual packing of bits
by groups of $8$ (called bytes) which is done in computers.
\\
This is why length considerations will always be developed
relative to binary alphabets.
A choice to be considered as a {\em normalization of length}.
\medskip\\
Finally, we come to the basic idea to measure the information
content of a mathematical object $x$ :
\medskip\\\medskip\centerline{\em
\begin{tabular}{|rcl|}
\hline
information content of $x$
&=&
\begin{tabular}{l}
length of a shortest binary word\\
which ``encodes" $x$
\end{tabular}
\\\hline
\end{tabular}}
What do we mean precisely by ``encodes" is the crucial question.
Following the trichotomy pointed by Kolmogorov \cite{kolmo65},
we survey three approaches.
%
%
%%%%%%%%%%%%%%%%%%%%%%%%%%%%%%%%%%%%%%%%%%%%%%%%%%%%%%%%%%%%%%%%%%%
\subsection{Combinatorial approach}\label{s:combinatorial}
%%%%%%%%%%%%%%%%%%%%%%%%%%%%%%%%%%%%%%%%%%%%%%%%%%%%%%%%%%%%%%%%%%%
%
%------------------------------------------------------------------
\subsubsection{Constant-length codes}\label{sss:constantlength}
%------------------------------------------------------------------
Let's consider the family $A^n$ of length $n$ words in an alphabet
$A$ with $s$ letters $a_1,...,a_s$.
Coding the $a_i$'s by binary words $w_i$'s all of length
$\lceil\log s\rceil$, to any word $u$ in $A^n$ we can associate
the binary word $\xi$ obtained by substituting the $w_i$'s to the
occurrences of the $a_i$'s in $u$.
Clearly, $\xi$ has length $n\lceil\log s\rceil$.
Also, the map $u\mapsto\xi$ is very simple. Mathematically,
it can be considered as a morphism from words in alphabet $A$
to binary words relative to the algebraic structure (of monoid)
given by the concatenation product of words.
\medskip\\
Observing that $n\log s$ can be smaller than $n \lceil\log s\rceil$,
a modest improvement is possible which saves about
$n\lceil \log s\rceil - n\log s$ bits.
The improved map $u\mapsto\xi$ is essentially a change of base:
looking at $u$ as the base $s$ representation of an integer
$k$, the word $\xi$ is the base $2$ representation of $k$.
Now, the map $u\mapsto\xi$ is no more a morphism.
However, it is still quite simple and can be computed
by a finite automaton.
{\small\begin{quote}
We have to consider $k$-adic representations rather than the usual
$k$-ary ones.
The difference is simple: instead of using digits $0,1,...,k-1$
use digits $1,...,k$. The interpretation as a sum of successive
exponentials of $k$ is unchanged and so are all usual algorithms
for arithmetical operations.
Also, the lexicographic ordering on $k$-adic representations
corresponds to the natural order on integers.
For instance, the successive integers $0,1,2,3,4,5,6,7$,
written $0$, $1$, $10$, $11$, $100$, $101$, $110$, $111$ in binary
(i.e. $2$-ary) have $2$-adic representations the empty word (for $0$)
and then the words $1$, $2$, $11$, $12$, $21$, $22$, $111$.
Whereas the length of the $k$-ary representation of $x$ is
$1+\lfloor\frac{\log x}{\log k}\rfloor$, its $k$-adic representation
has length $\lfloor\frac{\log (x+1)}{\log k}\rfloor$.
\\
Let's interpret the length $n$ word $u$ as the $s$-adic
representation of an integer $x$ between $t=s^{n-1}+...+s^2+s+1$
and $t'=s^n+...+s^2+s$
(which correspond to the length $n$ words $11...1$ and $ss...s$).
Let $\xi$ be the $2$-adic representation of this integer $x$.
The length of $\xi$ is
$\leq \lfloor\log (t'+1)\rfloor
   =  \lfloor \log(\frac{s^{n+1}-1}{s-1})\rfloor
 \leq  \lfloor (n+1)\log s - \log(s-1)\rfloor
 =  \lfloor n\log s - \log(1-\frac{1}{s})\rfloor$
which differs from $n\log s$ by at most $1$.
\end{quote}}
%
%------------------------------------------------------------------
\subsubsection{Variable-length prefix codes}\label{sss:varlength}
%------------------------------------------------------------------
Instead of coding the $s$ letters of $A$ by binary words of length
$\lceil\log s\rceil$, one can code the $a_i$'s by binary words $w_i$'s
having different lengthes so as to associate
short codes to most frequent letters and long codes to rare ones.
Which is the basic idea of compression.
Using these codes, the substitution of the $w_i$'s to the
occurrences of the $a_i$'s in a word $u$ gives a binary word $\xi$.
And the map $u\mapsto\xi$ is again very simple.
It is still a morphism from the monoid of words on alphabet $A$
to the monoid of binary words
and can also be computed by a finite automaton.
\\
Now, we face a problem: can we recover $u$ from $\xi$ ?
i.e. is the map $u\mapsto \xi$ injective?
In general the answer is no.
However, a simple sufficient condition to ensure decoding
is that the family $w_1,...,w_s$ be a so-called
{\em prefix-free code}.
Which means that if $i\neq j$ then $w_i$ is not a prefix of $w_j$.
{\small\begin{quote}
This condition insures that there is a unique $w_{i_1}$
which is a prefix of $\xi$. 
Then, considering the associated suffix $\xi_1$ of $v$
(i.e. $v=w_{i_1}\xi_1$) there is a unique $w_{i_2}$
which is a prefix of $\xi_1$, i.e. $u$ is of the form
$u=w_{i_1}w_{i_2}\xi_2$.
And so on.
\end{quote}}
\noindent
Suppose the numbers of occurrences in $u$ of the letters
$a_1,...,a_s$ are $m_1,...,m_s$,
so that the length of $u$ is $n=m_1+...+m_s$.
Using a prefix-free code $w_1,...,w_s$, the binary word $\xi$
associated to $u$ has length $m_1|w_1|+...+m_s|w_s|$.
A natural question is, given $m_1,...,m_s$,
{\em how to choose the prefix-free code $w_1,...,w_s$
so as to minimize the length of $\xi$ ?}
\\
David A. Huffman, 1952 \cite{huffman}, found a very efficient
algorithm (which has linear time complexity if the frequencies
are already ordered).
This algorithm (suitably modified to keep its top efficiency
for words containing long runs of the same data)
is nowadays used in nearly every application
that involves the compression and transmission of data:
fax machines, modems, networks,...
%
%------------------------------------------------------------------
\subsubsection{Entropy of a of distribution of frequencies}
\label{sss:entropy}
%------------------------------------------------------------------
The intuition of the notion of entropy in information theory
is as follows.
Given natural integers $m_1,...,m_s$, consider the family
$\+F_{m_1,...,m_s}$ of length $n=m_1+...+m_s$ words
of the alphabet $A$ in which there are exactly $m_1,...,m_s$
occurrences of letters $a_1,...,a_s$.
How many binary digits are there in the binary representation
of the number of words in $\+F_{m_1,...,m_s}$ ?
It happens (cf. Proposition \ref{p:H})
that this number is essentially linear in
$n$, the coefficient of $n$ depending solely on the frequencies
$\frac{m_1}{n},...,\frac{m_s}{n}$.
It is this coefficient which is called the entropy $H$ of the
distribution of the frequencies $\frac{m_1}{n},...,\frac{m_s}{n}$.
\\
Now, $H$ has a striking significance in terms of information
content and compression.
Any word $u$ in $\+F_{m_1,...,m_s}$ is uniquely
characterized by its rank in this family
(say relatively to the lexicographic ordering on words
in alphabet $A$).
In particular, the binary representation of this rank ``encodes"
$u$ and its length, which is bounded by $nH$
(up to an $O(\log n)$ term) can be seen as an upper bound
of the information content of $u$.
Otherwise said, the $n$ letters of $u$ are encoded by $nH$
binary digits.
In terms of compression
(nowadays so popularized by the zip-like softwares),
{\em $u$ can be compressed to $nH$ bits}
i.e.
{\em the mean information content
(which can be seen as the compression size in bits)
of a letter of $u$ is $H$}.
\begin{definition}[Shannon, 1948 \cite{shannon48}] 
Let $f_1,...,f_s$ be a distribution of frequencies,
i.e. a sequence of reals in $[0,1]$ such that $f_1+...+f_s=1$.
The entropy of $f_1,...,f_s$ is the real
$$
H = -(f_1\log(f_1) +...+f_s\log(f_s))
$$
\end{definition}
Let's look at two extreme cases.
\\
If all frequencies are equal to $\frac{1}{s}$
then the entropy is $\log(s)$, so that the mean information content
of a letter of $u$ is $\log(s)$,
i.e. there is no better  (prefix-free) coding than that described in
\S\ref{sss:constantlength}.
\\
In case one frequency is $1$ and the other ones are $0$,
the information content of $u$ is reduced to its length $n$,
which, written in binary, requires $\log(n)$ bits.
As for the entropy, it is $0$
(with the usual convention $0\log 0=0$, justified by the fact that
$\lim_{x\to0}x\log x=0$).
The discrepancy between $nH=0$ and the true information content
$\log n$ comes from the $O(\log n)$ term (cf. the next Proposition).
\begin{proposition}\label{p:H}
Let $m_1,...,m_s$ be natural integers and $n=m_1+...+m_s$.
Then, letting $H$ be the entropy of the distribution of frequencies
$\frac{m_1}{n},...,\frac{m_s}{n}$,
the number $\sharp\+F_{m_1,...,m_s}$ of words in
$\+F_{m_1,...,m_s}$ satisfies
$$
\log(\sharp\+F_{m_1,...,m_s}) = nH + O(\log n)
$$
where the bound in $O(\log n)$ depends solely on $s$
and not on $m_1,...,m_s$.
\end{proposition}
{\small\begin{quote}
\begin{proof}
$\+F_{m_1,...,m_s}$ contains
$\frac{n!}{m_1!\times...\times m_s!}$ words.
Using Stirling's approximation of the factorial function
(cf. Feller's book \cite{feller}), namely
$x! = \sqrt{2\pi} \ x^{x+\frac{1}{2}} \ e^{-x+\frac{\theta}{12}}$
where $0<\theta<1$ and equality $n=m_1+...+m_S$, we get
\begin{eqnarray*}
\log(\frac{n!}{m_1!\times...\times m_s!})
&=& (\sum_i m_i)\log(n) - (\sum_i m_i\log m_i)
\\&& 
+\frac{1}{2} \log(\frac{n}{m_1\times...\times m_s})
-(s-1)\log\sqrt{2\pi} + \alpha
\end{eqnarray*}
where $|\alpha|\leq\frac{s}{12}\log e$.
The first two terms are exactly
$n[\sum_i\frac{m_i}{n}\log(\frac{m_i}{n})] = nH$
and the remaining sum is $O(\log n)$ since
$n^{1-s} \leq \frac{n}{m_1\times...\times m_s} \leq 1$.
\end{proof}
\end{quote}}
%
%------------------------------------------------------------------
\subsubsection{Shannon's source coding theorem for symbol codes}
%------------------------------------------------------------------
The significance of the entropy explained above has been given
a remarkable and precise form by
Claude Elwood Shannon (1916-2001) in his celebrated 1948 paper
\cite{shannon48}.
It's about the length of the binary word $\xi$ associated
to $u$ via a prefix-free code.
Shannon proved
\\- a lower bound of $|\xi|$ valid whatever be
the prefix-free code $w_1,...,w_s$,
\\- an upper bound, quite close to the lower bound,
valid for particular prefix-free codes $w_1,...,w_s$
(those making $\xi$ shortest possible, for instance those given
by Huffman's algorithm).
\begin{theorem}[Shannon, 1948 \cite{shannon48}]
Suppose the numbers of occurrences in $u$ of the letters
$a_1,...,a_s$ are $m_1,...,m_s$. Let $n=m_1+...+m_s$.
\medskip\\
1. For every prefix-free sequence of binary words $w_1,...,w_s$,
the binary word $\xi$ obtained by substituting $w_i$ to each
occurrence of $a_i$ in $u$
satisfies
$$
nH \leq |\xi|
$$
where
$H\ =\ -(\frac{m_1}{n}\log(\frac{m_1}{n}) +...
+ \frac{m_s}{n}\log(\frac{m_s}{n}))$
is the so-called entropy of the considered distribution
of frequencies $\frac{m_1}{n},...,\frac{m_s}{n}$.
\medskip\\
2. There exists a prefix-free sequence of binary words $w_1,...,w_s$
such that
$$
nH \leq |\xi| < n(H+1)
$$
\end{theorem}
{\small\begin{quote}
\begin{proof}
First, we recall two classical results.
\begin{theorem*}[Kraft's inequality]
Let $\ell_1,...,\ell_s$ be a finite sequence of integers.
Inequality $2^{-\ell_1}+...+2^{-\ell_s} \leq 1$ holds
if and only if there exists a prefix-free sequence of binary words
$w_1,...,w_s$ such that $\ell_1=|w_1|,...,\ell_s=|w_s|$.
\end{theorem*}
\begin{theorem*}[Gibbs' inequality]
Let $p_1,...,p_s$ and $q_1,...,q_s$ be two probability
distributions, i.e. the $p_i$'s (resp. $q_i$'s) are in $[0,1]$
and have sum $1$.
Then
$-\sum p_i\log(p_i) \leq -\sum p_i\log(q_i)$
with equality if and only if $p_i=q_i$ for all $i$.
\end{theorem*}
{\em Proof of 1}.
Set $p_i=\frac{m_i}{n}$ and $q_i=\frac{2^{-|w_i|}}{S}$
where $S=\sum_i 2^{-|w_i|}$. Then
\medskip\\
$|\xi| = \sum_i m_i|w_i|
=n[\sum_i\frac{m_i}{n}(-\log(q_i) -\log S)]$

\hfill{$\geq n[-(\sum_i\frac{m_i}{n}\log(\frac{m_i}{n}) -\log S]
=n[H-\log S]
\geq nH$}
\medskip\\
The first inequality is an instance of Gibbs' inequality.
For the last one, observe that $S\leq1$ and apply Kraft' inequality.
\medskip\\
{\em Proof of 2}.
Set $\ell_i=\lceil-\log(\frac{m_i}{n})\rceil$.
Observe that $2^{-\ell_i} \leq \frac{m_i}{n}$.
Thus, $2^{-\ell_1}+...+2^{-\ell_s} \leq 1$.
Applying Kraft inequality, we see that there exists a prefix-free
family of words $w_1,...,w_s$ with lengthes $\ell_1,...,\ell_s$.
\\
We consider the binary word $\xi$ obtained via this prefix-free
code, i.e. $\xi$ is obtained by substituting $w_i$ to each
occurrence of $a_i$ in $u$.
Observe that
$-\log(\frac{m_i}{n}) \leq \ell_i < -\log(\frac{m_i}{n})+1$.
Summing, we get $nH \leq |\xi| \leq  n(H+1)$.
\end{proof}
\end{quote}}
\noindent
In particular cases, the lower bound $nH$ is exactly $|\xi|$.
\begin{theorem}
In case the frequencies $\frac{m_i}{n}$'s are all negative
powers of two (i.e. $\frac{1}{2},\frac{1}{4},\frac{1}{8}$,...)
then the optimal $\xi$ (given by Huffman algorithm)
satisfies $\xi=nH$.
\end{theorem}
%
%------------------------------------------------------------------
\subsubsection{Closer to the entropy}\label{sss:closer}
%------------------------------------------------------------------
As simple as they are, prefix-free codes are not the only way
to efficiently encode into a binary word $\xi$ a word $u$ from
alphabet $a_1,...,a_s$ for which the numbers $m_1,...,m_s$
(of occurrences of the $a_i$'s) are known.
Let's go back to the encoding mentioned at the start of
\S\ref{sss:entropy}.
A word $u$ in the family $\+F_{m_1,...,m_s}$
(of length $n$ words with exactly $m_1,...,m_s$ occurrences
of $a_1,...,a_s$)
can be recovered from the following data:
\\- the values of $m_1,...,m_s$,
\\- the rank of $u$ in $\+F_{m_1,...,m_s}$
    (relative to the lexicographic order on words).
\\
We have seen (cf. Proposition \ref{p:H})
that the rank of $u$ has a binary representation $\rho$
of length $\leq nH+O(\log n)$.
The integers $m_1,...,m_s$ are encoded by their binary
representations $\mu_1,...,\mu_s$ which are all
$\leq 1+\lfloor\log n\rfloor$.
Now, to encode $m_1,...,m_s$ and the rank of $u$, we cannot
just concatenate $\mu_1,...,\mu_s,\rho$ : how would we know
where $\mu_1$ stops, where $\mu_2$ starts,...,
in the word obtained by concatenation?
Several tricks are possible to overcome the problem,
they are described in \S\ref{sss:codemany}.
Using Proposition \ref{p:code}, we set
$\xi = \langle \mu_1,...,\mu_s,\rho \rangle$
which has length
$|\xi| = |\rho|+O(|\mu_1|+...+|\mu_s|) = nH +O(\log n)$
(Proposition \ref{p:code} gives a much better bound but this is
of no use here).
Then, $u$ can be recovered from $\xi$ which is a binary word
of length $nH+O(\log n)$.
Thus, asymptotically, we get a better upper bound than $n(H+1)$,
the one given by Shannon for codings with prefix-free codes.
\\
Of course, $\xi$ is no more obtained from $u$ via a morphism
(i.e. a map which preserves concatenation of words)
between the monoid of words in alphabet $A$ to that of binary words.
%
%
%------------------------------------------------------------------
\subsubsection{Coding finitely many words with one word}
\label{sss:codemany}
%------------------------------------------------------------------
How can we code two words $u,v$ by one word?
The simplest way is to consider $u\$ v$ where $\$ $
is a fresh symbol outside the alphabet of $u$ and $v$.
But what if we want to stick to binary words?
As said above, the concatenation of $u$ and $v$ does not
do the job: one cannot recover the right prefix $u$ in $uv$.
A simple trick is to also concatenate the length of $|u|$ in
unary and delimitate it by a zero:
indeed, from the word $1^{|u|}0uv$ one can recover $u$ and $v$.
In other words, the map $(u,v)\to 1^{|u|}0uv$ is injective
from $\words\times\words\to\words$.
In this way, the code of the pair $(u,v)$ has length $2|u|+|v|+1$.
\\
This can obviously be extended to more arguments.
\begin{proposition}\label{p:code}
Let $s\geq1$.
There exists a map
$\langle\ \rangle : (\words)^{s+1}\to\words$
which is injective and computable and such that,
for all $u_1,...,u_s,v\in\words$,
$|\langle u_1,...,u_s,v \rangle|
= 2(|u_1|+...+|u_s|)+|v|+ s$.
\end{proposition}
This can be improved, we shall need this technical improvement
in \S\ref{ss:nid}.
\begin{proposition}\label{p:codeloglog}
There exists an injective and computable such that,
for all $u_1,...,u_s,v\in\words$,
\begin{multline*}
|\langle u_1,...,u_s,v \rangle|
= (|u_1|+...+|u_s|+|v|)
+ (\log|u_1|+...+\log|u_s|)\\
+ O((\log\log|u_1|+...+\log\log|u_s|))
$$
\end{multline*}
\end{proposition}

{\small\begin{quote}
\begin{proof}
We consider the case $s=1$, i.e. we want to code a pair $(u,v)$.
Instead of putting the prefix $1^{|u|}0$, let's put the binary
representation $\beta(|u|)$ of the number $|u|$ prefixed by
its length.
This gives the more complex code:
$1^{|\beta(|u|)|}0\beta(|u|)uv$
with length
$$
|u|+|v|+2(\lfloor\log|u|\rfloor+1)+1
\leq |u|+|v|+2\log|u| + 3
$$
The first block of ones gives the length of $\beta(|u|)$.
Using this length, we can get $\beta(|u|)$ as the factor following
this first block of ones.
Now, $\beta(|u|)$ is the binary representation of $|u|$, so we get
$|u|$ and can now separate $u$ and $v$ in the suffix $uv$.
\end{proof}
\end{quote}}
%
%
%%%%%%%%%%%%%%%%%%%%%%%%%%%%%%%%%%%%%%%%%%%%%%%%%%%%%%%%%%%%%%%%%%%
\subsection{Probabilistic approach}
%%%%%%%%%%%%%%%%%%%%%%%%%%%%%%%%%%%%%%%%%%%%%%%%%%%%%%%%%%%%%%%%%%%
%
The abstract probabilistic approach allows for considerable
extensions of the results described in \S\ref{s:combinatorial}.
\medskip\\
First, the restriction to fixed given frequencies can be relaxed.
The probability of writing $a_i$ may depend on what has
been already written. For instance, Shannon's source coding theorem
has been extended to the so called
``ergodic asymptotically mean stationary source models".
\medskip\\
Second, one can consider a lossy coding: some length $n$ words in
alphabet $A$ are ill-treated or ignored.
Let $\delta$ be the probability of this set of words.
Shannon's theorem extends as follows:
\\- whatever close to $1$ is $\delta<1$,
one can compress $u$ only down to $nH$ bits.
\\- whatever close to $0$ is $\delta>0$,
one can achieve compression of $u$ down to $nH$ bits.
%
%
%
%%%%%%%%%%%%%%%%%%%%%%%%%%%%%%%%%%%%%%%%%%%%%%%%%%%%%%%%%%%%%%%%%%%
\subsection{Algorithmic approach}
%%%%%%%%%%%%%%%%%%%%%%%%%%%%%%%%%%%%%%%%%%%%%%%%%%%%%%%%%%%%%%%%%%%
%%
%
%------------------------------------------------------------------
\subsubsection{Berry's paradox}\label{sss:berry}
%------------------------------------------------------------------
So far, we considered two kinds of binary codings for a word
$u$ in alphabet $a_1,...,a_s$.
The simplest one uses variable-length prefix-free codes
(\S\ref{sss:varlength}).
The other one codes the rank of $u$ as a member of some set
(\S\ref{sss:closer}).
\\
Clearly, there are plenty of other ways to encode any
mathematical object.
Why not consider all of them? And define the information
content of a mathematical object $x$ as
{\em the shortest univoque description of $x$ (written as
a binary word)}.
Though quite appealing, this notion is ill defined
as stressed by Berry's paradox\footnote{
Berry's paradox is mentioned by Bertrand Russell, 1908
(\cite{russell}, p.222 or 150) who credited G.G. Berry,
an Oxford librarian, for the suggestion.}:
\begin{quote}
Let $\beta$ be the
{\em lexicographically least binary word which cannot be
univoquely described by any binary word of length less
than $1000$}.
\end{quote}
This description of $\beta$ contains $106$ symbols of written
English (including spaces) and, using ASCII codes,
can be written as a binary word of length $106\times7=742$.
Assuming such a description to be well defined
would lead to a univoque description of $\beta$ in $742 $ bits,
hence less than $1000$,
a contradiction to the definition of $\beta$.
\\
The solution to this inconsistency is clear:
the quite vague notion of univoque description entering
Berry's paradox is used both inside the sentence
describing $\beta$ and inside the argument to get the
contradiction. A collapse of two levels:
\\\indent- the would be formal level carrying the description
           of $\beta$
\\ \indent- and the meta level which carries the inconsistency
            argument.
\\
Any formalization of the notion of description
should drastically reduce its scope and totally forbid
the above collapse.
%
%------------------------------------------------------------------
\subsubsection{The turn to computability}\label{sss:turn}
%------------------------------------------------------------------
To get around the stumbling block of Berry's paradox
and have a formal notion of description with wide scope,
Andrei Nikolaievitch Kolmogorov (1903--1987) made an
ingenious move: he turned to computability and
replaced {\em description} by {\em computation program}.
Exploiting the successful formalization of this a priori
vague notion which was achieved in the thirties\footnote{
Through the works of Alonzo Church (via lambda calculus),
Alan Mathison Turing (via Turing machines)
and Kurt G\"odel and Jacques Herbrand (via Herbrand-G\"odel
systems of equations)
and Stephen Cole Kleene (via the recursion and minimization
operators).}.
This approach was first announced by Kolmogorov in
\cite{kolmo63}, 1963,
and then developped in \cite{kolmo65}, 1965.
Similar approaches were also independently developped by
Ray J. Solomonoff in \cite{solo64a,solo64b}, 1964,
and by Gregory Chaitin in \cite{chaitin66, chaitin69}, 1966-69.
%
%
%------------------------------------------------------------------
\subsubsection{Digression on computability theory}
\label{sss:partial}
%------------------------------------------------------------------
%
The formalized notion of {\em computable function}
(also called recursive function)
goes along with that of {\em partial computable function}
which should rather be called
{\em partially computable partial function}
(also called partial recursive function),
i.e. the {\em partial} qualifier has to be
distributed.\footnote{In French, Daniel Lacombe used the
expression {\em semi-fonction semi-r\'ecursive}}
\\
So, there are two theories :
\\\indent- {\em the theory of  computable functions},
\\\indent- {\em the theory of partial computable functions}.
\\
The ``right" theory, the one with a cornucopia of spectacular
results, is that of partial computable functions.
{\small\begin{quote}
Let's pick up three fundamental results out of the cornucopia,
which we state in terms of computers and programming languages.
Let $\+I$ and $\+O$ be non empty finite products of
simple countable families of mathematical objects such as
$\N$, $A^*$ (the family of words in alphabet $A$)
where $A$ is finite or countably infinite.
\begin{theorem}\label{thm:3thms}
{\em 1. [Enumeration theorem]}
The {\em(program, input) $\to$ output} function
which executes programs on their inputs
is itself partial computable.
\\
Formally, this means that there exists a partial computable
function
$$
U: \words\times\+I \to \+O
$$
such that the family of partial computable function
$\+I \to \+O$ is exactly
$\{U_e \mid e\in \words\}$ where $U_e(x)=U(e,x)$.
\\
Such a function $U$ is called universal for partial
computable functions $\+I \to \+O$.
\medskip\\
{\em 2. [Parameter theorem (or $s^m_n$ thm)].}
One can exchange input and program
{\em(this is von Neumann's key idea for computers)}.
\\
Formally, this means that, letting $\+I=\+I_1\times \+I_2$,
universal maps $U_{\+I_1\times \+I_2}$ and $U_{\+I_2}$
are such that there exists a computable total map
$s: \words\times \+I_1 \to \words$
such that, for all $e\in \words$,
$x_1\in \+I_1$ and $x_2\in \+I_2$,
$$
U_{\+I_1\times \+I_2}(e,(x_1,x_2)) = U_{\+I_2}(s(e,x_1),x_2)
$$
{\em 3. [Kleene fixed point theorem]}
For any transformation of programs, there is a program
which does the same {\em input $\to$ output} job
as its transformed program.
(Note: This is the seed of computer virology...
cf. \cite{BKM06} 2006)
\\
Formally, this means that, for every partial computable map
$f: \words\to \words$, there exists $e$ such that
$$
\forall e\in \words\ \forall x\in\+I\ \ \
U(f(e),x)=U(e,x)
$$
\end{theorem}
\end{quote}}
%
%%%%%%%%%%%%%%%%%%%%%%%%%%%%%%%%%%%%%%%%%%%%%%%%%%%%%%%%%%%%%%%%%%%
\subsection{Kolmogorov complexity and the invariance theorem}
\label{ss:invariance}
%%%%%%%%%%%%%%%%%%%%%%%%%%%%%%%%%%%%%%%%%%%%%%%%%%%%%%%%%%%%%%%%%%%
%
{\small\begin{quote}
\begin{note*}
The denotations of (plain) Kolmogorov complexity
and its prefix version may cause some confusion.
They long used to be respectively denoted by $K$ and $H$
in the literature.
But in their book \cite{livitanyi}, Li \& Vitanyi respectively
denoted them $C$ and $K$.
Due to the large success of this book, these last denotations
are since used in many papers.
So that two incompatible denotations now appear in the
litterature.
Since we mainly focus on plain Kolmogorov complexity,
we stick to the traditional denotations $K$ and $H$.
\end{note*}
\end{quote}}
%
%------------------------------------------------------------------
\subsubsection{Program size complexity
or Kolmogorov complexity}
%------------------------------------------------------------------
Turning to computability, the basic idea for Kolmogorov complexity
is
\medskip\\\medskip\centerline{\em\begin{tabular}{|rcl|}
\hline
description &=& program
\\\hline
\end{tabular}}
When we say ``program", we mean a program taken from a family
of programs, i.e. written in a programming language or describing
a Turing machine or a system of Herbrand-G\"odel equations
or a Post system,...
\\
Since we are soon going to consider length of programs,
following what has been said in \S\ref{sss:words},
we normalize programs: they will be binary words,
i.e. elements of $\words$.
\\
So, we have to fix a function $\varphi:\words\to\+O$ and
consider that the output of a program $p$ is $\varphi(p)$.
\\
Which $\varphi$ are we to consider? Since we know that there are
universal partial computable functions
(i.e. functions able to emulate any other partial computable
function modulo a computable transformation of programs,
in other words, a compiler from one language to another),
it is natural to consider universal partial computable functions.
Which agrees with what has been said in \S\ref{sss:partial}.
\\
The general definition of the Kolmogorov complexity
associated to any function $\words\to\+O$ is as follows.
\begin{definition}\label{def:Kphi}
If $\varphi:\words\to\+O$ is a partial function,
set $K_\varphi : \+O \to \N$
$$
K_\varphi(y)=\min\{|p| : \varphi(p)=y\}
$$
Intuition:
$p$ is a program (with no input),
$\varphi$ executes programs
(i.e. $\varphi$ is all together
a programming language
plus a compiler
plus a machinery to run programs)
and $\varphi(p)$ is the output of the run of program $p$.
Thus, for $y\in\+O$, $K_\varphi(y)$ is the length of shortest
programs $p$ with which $\varphi$ computes $y$
(i.e. $\varphi(p)=y$)
\end{definition}
As said above, we shall consider this definition for
partial computable functions $\words\to\+O$.
Of course, this forces to consider a set $\+O$ endowed with
a computability structure.
Hence the choice of sets that we shall call {\em elementary}
which do not exhaust all possible ones but will suffice for
the results mentioned in this paper.
\begin{definition}\label{def:elementary}
The family of elementary sets is obtained as follows:
\\- it contains $\N$ and the $A^*$'s where $A$ is a finite or
countable alphabet,
\\-  it is closed under finite (non empty) product,
product with any non empty finite set and
the finite sequence operator
\end{definition}
%
%------------------------------------------------------------------
\subsubsection{The invariance theorem}
%------------------------------------------------------------------
The problem with Definition \ref{def:Kphi} is that $K_\varphi$
strongly depends on $\varphi$.
Here comes a remarkable result, the invariance theorem,
which insures that {\em there is a smallest $K_\varphi$
up to a constant}.
It turns out that the proof of this theorem only needs
the enumeration theorem and makes no use of the parameter
theorem (usually omnipresent in computability theory).
\begin{theorem}
[Invariance theorem, Kolmogorov, \cite{kolmo65},1965]
\label{thm:invariance}
Let $\+O$ be an elementary set.
Among the $K_\varphi$'s, where $\varphi:\words\to\+O$
varies in the family $PC^\+O$ of partial computable functions,
there is a smallest one, up to an additive constant
(= within some bounded interval). I.e.
$$
\exists V\in PC^\+O\ \forall \varphi\in PC^\+O\ \exists c\
\forall y\in\+O\ \ K_V(y) \leq K_\varphi(y) + c
$$
Such a $V$ is called optimal.
\end{theorem}
{\small\begin{quote}
\begin{proof}
Let $U:\words\times \words\to\+O$ be a partial computable
universal function for partial computable
functions $\words\to\+O$
(cf. Theorem \ref{thm:3thms}, Enumeration theorem).
\\
Let $c: \words\times\words\to \words $ be a total
computable injective map such that $|c(e,x)|=2|e|+|x|$
(cf. Proposition \ref{p:code}).
\\
Define $V: \words\to\+O$ as follows:
$$
\forall e\in\words\ \forall x\in \words\ \
V(c(e,x)) = U(e,x)
$$
where equality means that both sides are simultaneously
defined or not.
Then, for every partial computable function
$\varphi: \words\to\+O$, for every $y\in\+O$,
if $\varphi=U_e$
(i.e. $\varphi(x)=U(e,x)$ for all $x$,
cf. Theorem \ref{thm:3thms}, Enumeration theorem)
then
\begin{eqnarray*}
K_V(y) &=& \mbox{least $|p|$ such that $V(p)=y$}
\\
&\leq& \mbox{least $|c(e,x)|$ such that $V(c(e,x))=y$}
\\
&=& \mbox{least $|c(e,x)|$ such that $U(e,x))=y$}
\\
&=& \mbox{least $|x|+2|e|+1$ such that $\varphi(x)=y$}
\\&&\hspace{1cm}
\mbox{since $|c(e,x)|=|x|+2|e|+1$ and $\varphi(x)=U(e,x)$}
\\
& =& (\mbox{least $|x|$ such that $\varphi(x)=y$})+2|e|+1
\\
&=& K_\varphi(y)+2|e|+1
\end{eqnarray*}
\end{proof}
\end{quote}}
Using the invariance theorem, the Kolmogorov complexity
$K:\+O\to\N$ is defined as $K_V$ where $V$ is any fixed optimal function.
The arbitrariness of the choice of $V$ does not modify
drastically $K_V$, merely up to a constant.
\begin{definition}
Kolmogorov complexity $K^\+O:\+O\to\N$ is
$K_\varphi$ where $\varphi$ is some fixed
optimal function $\+I\to\+O$.
$K^\+O$ will be denoted by $K$ when $\+O$ is clear from context.
\\
$K^\+O$ is therefore minimum among the $K_\varphi$'s,
up to an additive constant.
\medskip\\
$K^\+O$ is defined up to an additive constant:
if $V$ and $V'$ are both optimal then
$$
\exists c\ \forall x\in\+O\ |K_V(x) - K_{V'}(x)| \leq c
$$
\end{definition}
%
%
%------------------------------------------------------------------
\subsubsection{About the constant}
%------------------------------------------------------------------
So Kolmogorov complexity is an integer defined up to
a constant\ldots !
But the constant is uniformly bounded for $x\in\+O$.
\\
Let's quote what Kolmogorov said about the constant
in \cite{kolmo65}:
\begin{quote}\em
Of course, one can avoid the indeterminacies
associated with the [above] constants,
by considering particular [\ldots functions $V$],
but it is doubtful that this can be done without
explicit arbitrariness.
\\
One must, however, suppose that the different
``reasonable" [above optimal functions] will lead to
``complexity estimates"
that will converge on hundreds of bits
instead of tens of thousands.
\\
Hence, such quantities as the ``complexity" of the text of
``War and Peace" can be assumed to be defined
with what amounts to uniqueness.
\end{quote}
In fact, this constant is in relation with the multitude
of models of computation: universal Turing machines,
universal cellular automata,
Herbrand-G\"odel systems of equations, Post systems,
KLeene definitions,...
If we feel that one of them is canonical then we may consider
the associated Kolmogorov complexity as the right one
and forget about the constant.
This has been developed for Schoenfinkel-Curry combinators
$S,K,I$ by Tromp \cite{livitanyi} \S3.2.2--3.2.6.
\\
However, this does absolutely not lessen the importance
of the invariance theorem since it tells us that $K$ is less
than {\em any} $K_\varphi$ (up to a constant).
A result which is applied again and again to develop the theory.
%
%
%------------------------------------------------------------------
\subsubsection{Conditional Kolmogorov complexity}
%------------------------------------------------------------------
In the enumeration theorem (cf. Theorem \ref{thm:3thms}),
we considered {\em(program, input) $\to$ output} functions.
Then, in the definition of Kolmogorov complexity, we gave up
the inputs, dealing with functions {\em program $\to$ output}.
\\
Conditional Kolmogorov complexity deals with the inputs. 
Instead of measuring the information content of $y\in\+O$,
we measure it given as free some object $z$, which may help
to compute $y$.
A trivial case is when $z=y$, then the information content
of $y$ given $y$ is null. In fact, there is an obvious program
which outputs exactly its input, whatever be the input.
\\
Let's mention that, in computer science, inputs are also
considered as {\em the environment}.
\\
Let's give the formal definition and the adequate
invariance theorem.
\begin{definition}\label{def:condKphi}
If $\varphi:\words\times\+I\to\+O$ is a partial function,
set $K_\varphi(\ \mid\ ) : \+O\times\+I \to \N$
$$
K_\varphi(y \mid z)=\min\{|p| \mid \varphi(p,z)=y\}
$$
Intuition:
$p$ is a program (with no input),
$\varphi$ executes programs
(i.e. $\varphi$ is all together
a programming language
plus a compiler
plus a machinery to run programs)
and $\varphi(p,z)$ is the output of the run of program $p$
on input $z$.
Thus, for $y\in\+O$, $K_\varphi(y)$ is the length of shortest
programs $p$ with which $\varphi$ computes $y$ on input $z$
(i.e. $\varphi(p,z)=y$)
\end{definition}
\begin{theorem}[Invariance theorem for conditional complexity]
Among the $K_\varphi(\ |\ )$'s, where $\varphi $ varies
in the family $PC^\+O_\+I$
of partial computable function $\words\times \+I\to\+O$,
there is a smallest one, up to an additive constant
(i.e. within some bounded interval) :
$$
\exists V\in PC^\+O_\+I\ \forall \varphi\in PC^\+O_\+I\
\exists c\ \forall y\in\+O\ \forall z\in\+I\ \
K_V(y\mid z) \leq K_\varphi(x\mid y) + c
$$
Such a $V$ is called optimal.
\end{theorem}
\begin{proof}
Simple application of the enumeration theorem
for partial computable functions.
\end{proof}
\begin{definition}
$K^\+O_\+I:\+O\times\+I\to\N$ is $K_V(\ |\ )$ where $V$
is some fixed optimal function.
\medskip\\
$K^\+O_\+I$ is defined up to an additive constant:
if $V$ et $V'$ are both minimum then
$$
\exists c\ \forall y\in\+O\ \forall z\in\+I\
|K_V(y\mid z) - K_{V'}(y\mid z)| \leq c
$$
\end{definition}
\noindent
Again, an integer defined up to a constant\ldots!
However, the constant is uniform in $y\in\+O$ and $z\in\+I$.
%
%
%------------------------------------------------------------------
\subsubsection{Simple upper bounds for Kolmogorov complexity}
%------------------------------------------------------------------
%
Finally, let's mention rather trivial upper bounds:
\\- the information content of a word is at most its length.
\\- conditional complexity cannot be harder than the non conditional
one.
\begin{proposition}\label{p:bound}
1. There exists $c$ such that,
$$
\forall x\in\words\ \ K^{\words}(x)\leq |x|+c
\ \ ,\ \
\forall n\in\N\ \ K^\N(n)\leq \log(n)+c
$$
2. There exists $c$ such that,
$$
\forall x\in D\ \forall y\in E\ \ \ K(x\mid y)\leq K(x)+c
$$
3. Let $f:\+O\to\+O'$ be computable.
Then, $K^{\+O'}(f(x)) \leq K^\+O(x) +O(1)$.
\end{proposition}
{\small\begin{quote}
\begin{proof}
We only prove 1.
Let $Id:\words\to\words$ be the identity function.
The invariance theorem insures that there exists $c$ such that
$K^{\words} \leq K^{\words}_{Id} + c$.
In particular,
for all $x\in\words$, $K^{\words}(x)\leq |x|+c$.
\\
Let $\theta:\words\to\N$ be the function which associate to a word
$u=a_{k-1}...a_0$ the integer
$$
\theta(u) = (2^k+a_{k-1}2^{k-1}+...+2a_1+a_0) -1
$$
(i.e. the predecessor of the integer with binary representation $1u$).
Clearly,
$K^\N_\theta(n)=\lfloor \log(n+1)\rfloor$.
The invariance theorem insures that there exists $c$ such that
$K^\N \leq K^\N_\theta + c$. Hence $K^\N(n)\leq \log(n)+c+1$
for all $n\in\N$.
\end{proof}
\end{quote}}
The following property is a variation of an argument
already used in \S\ref{sss:closer}: the rank of an element
in a set defines it, and if the set is computable,
so is this process.
\begin{proposition}\label{p:rank}
Let $A\subseteq \N\times D$ be computable such that
$A_n = A\cap(\{n\}\times D)$ is finite for all $n$.
Then, letting $\sharp(X)$ be the number of elements of $X$,
$$
\exists c\ \forall x\in A_n\ \ \
K(x\mid n) \leq \log(\sharp(A_n))+c
$$
Intuition. An element in a set is determined by its rank.
And this is a computable process.
\end{proposition}
{\small\begin{quote}
\begin{proof}
Observe that $x$ is determined by its rank in $A_n$.
This rank is an integer $<\sharp A_n$ hence with binary
representation of length $\leq\lfloor\log(\sharp A_n)\rfloor+1$.
\end{proof}
\end{quote}}
%
%
%%%%%%%%%%%%%%%%%%%%%%%%%%%%%%%%%%%%%%%%%%%%%%%%%%%%%%%%%%%%%%%%%%%
\subsection{Oracular Kolmogorov complexity }
%%%%%%%%%%%%%%%%%%%%%%%%%%%%%%%%%%%%%%%%%%%%%%%%%%%%%%%%%%%%%%%%%%%
%
As is always the case in computability theory, everything
relativizes to any oracle $Z$. This means that the equation
given at the start of \S\ref{ss:invariance} now becomes
\medskip\\ \medskip\centerline{\em\begin{tabular}{rcl}
description &=& program of a partial $Z$-computable function
\end{tabular}}
and for each possible oracle $Z$ there exists a Kolmogorov
complexity relative to oracle $Z$.
\medskip\\
Oracles in computability theory can also be considered as
second-order arguments of computable or partial computable
{\em functionals}.
The same holds with oracular Kolmogorov complexity:
the oracle $Z$ can be seen as a second-order condition
for a {\em second-order conditional Kolmogorov complexity}
$$
K(y\mid Z)\hspace{3mm}\mbox{where}\hspace{3mm}
K(\ \mid\ ):\+O \times P(\+I) \to \N
$$
Which has the advantage that the unavoidable constant in the
``up to a constant" properties does not depend on the
particular oracle. It depends solely on the considered
functional.
\\
Finally, one can mix first-order and second-order conditions,
leading to a conditional Kolmogorov complexity with both
first-order and second-order conditions
$$
K(y\mid z, Z)\hspace{3mm}\mbox{where}\hspace{3mm}
K(\ \mid\ ,\ ):\+O \times \+I\times P(\+I) \to \N
$$
We shall see in \S\ref{sss:nies}
an interesting property involving oracular Kolmogorov complexity. 
%
%
%
%
%
%%%%%%%%%%%%%%%%%%%%%%%%%%%%%%%%%%%%%%%%%%%%%%%%%%%%%%%%%%%%%%%%%%%
%%%%%%%%%%%%%%%%%%%%%%%%%%%%%%%%%%%%%%%%%%%%%%%%%%%%%%%%%%%%%%%%%%%
%%%%%%%%%%%%%%%%%%%%%%%%%%%%%%%%%%%%%%%%%%%%%%%%%%%%%%%%%%%%%%%%%%%
\section{Kolmogorov complexity and undecidability}
%%%%%%%%%%%%%%%%%%%%%%%%%%%%%%%%%%%%%%%%%%%%%%%%%%%%%%%%%%%%%%%%%%%
%%%%%%%%%%%%%%%%%%%%%%%%%%%%%%%%%%%%%%%%%%%%%%%%%%%%%%%%%%%%%%%%%%%
%%%%%%%%%%%%%%%%%%%%%%%%%%%%%%%%%%%%%%%%%%%%%%%%%%%%%%%%%%%%%%%%%%%
%
%%%%%%%%%%%%%%%%%%%%%%%%%%%%%%%%%%%%%%%%%%%%%%%%%%%%%%%%%%%%%%%%%%%
\subsection{$K$ is unbounded}
%%%%%%%%%%%%%%%%%%%%%%%%%%%%%%%%%%%%%%%%%%%%%%%%%%%%%%%%%%%%%%%%%%%
Let $K=K_V : \+O\to\N$ where $V:\words\to\+O$ is optimal
(cf. Theorem \S\ref{thm:invariance}).
Since there are finitely many programs of size $\leq n$
(namely $2^{n+1}-1$ words),
there are finitely many elements of $\+O$ with Kolmogorov complexity
less than $n$.
This shows that $K$ is unbounded.
%
%%%%%%%%%%%%%%%%%%%%%%%%%%%%%%%%%%%%%%%%%%%%%%%%%%%%%%%%%%%%%%%%%%%
\subsection{$K$ is not computable}\label{ss:Knoncomput}
%%%%%%%%%%%%%%%%%%%%%%%%%%%%%%%%%%%%%%%%%%%%%%%%%%%%%%%%%%%%%%%%%%%
Berry' paradox (cf. \S\ref{sss:berry}) has a counterpart in terms
of Kolmogorov complexity, namely it gives a proof that $K$,
which is a total function $\+O\to\N$, is not computable.
{\small\begin{quote}
\begin{proof}
For simplicity of notations, we consider the case $\+O=\N$.
Define $L:\N\to\+O$ as follows:
\begin{eqnarray*}
L(n) &=& \mbox{least  $k$ such that $K(k)\geq 2n$}
\end{eqnarray*}
So that $K(L(n))\geq n$ for all $n$.
If $K$ were computable so would be $L$.
Let $V:\+O\to\N$ be optimal, i.e. $K=K_V$.
The invariance theorem insures that there exists $c$ such that
$K\leq K_L+c$. Observe that $K_L(L(n)\leq n$ by definition of
$K_L$.
Then
$$
2n \leq K(L(n)) \leq K_L(L(n)+c \leq n+c
$$
A contradiction for $n>c$.
\end{proof}
\end{quote}}
The undecidability of $K$ can be sen as a version of the
undecidability of the halting problem.
In fact, there is a simple way to compute $K$ when the halting
problem is used as an oracle.
To get the value of $K(x)$, proceed as follows:
\\\indent- enumerate the programs in $\words$ in lexicographic
          order,
\\\indent- for each program $p$ check if $V(p)$ halts
          (using the oracle),
\\\indent- in case $V(p)$ halts then compute its value,
\\\indent- halt and output $|p|$ when some $p$ is obtained
           such that $V(p)=x$.
\medskip\\
The argument for the undecidability of $K$ can be used to prove
a much stronger statement: $K$ can not be bounded from below by
an unbounded partial computable function.
\begin{theorem}[Kolmogorov]
There is no unbounded partial recursive function
$\varphi:\+O\to\N$ such that $\varphi(x)\leq K(x)$ for all $x$
in the domain of $\varphi$.
\end{theorem}
Of course, $K$ is bounded from above by a total computable
function, cf. Proposition \ref{p:bound}.
%
%%%%%%%%%%%%%%%%%%%%%%%%%%%%%%%%%%%%%%%%%%%%%%%%%%%%%%%%%%%%%%%%%%%
\subsection{$K$ is computable from above}\label{ss:Kabove}
%%%%%%%%%%%%%%%%%%%%%%%%%%%%%%%%%%%%%%%%%%%%%%%%%%%%%%%%%%%%%%%%%%%
%
Though $K$ is not computable, it can be approximated from above.
The idea is simple. Suppose $\+O=\words$.
consider all programs of length less than $|x|$ and let them be
executed during $t$ steps.
If none of them converges and outputs $x$
then the $t$-bound is $|x|$.
If some of them converges and outputs $x$
then the bound is the length of the shortest such program.
\\
The limit of this process is $K(x)$, it is obtained at some
finite step which we are not able to bound.
\\
Formally, this means that there is some $F:\+O\times\N\to\N$
which is computable and decreasing in its second argument
such that
$$
K(x) = \lim_{n\to+\infty} F(x,n)
$$
%
%
%%%%%%%%%%%%%%%%%%%%%%%%%%%%%%%%%%%%%%%%%%%%%%%%%%%%%%%%%%%%%%%%%%%
\subsection{Kolmogorov complexity
            and G\"odel's incompleteness theorem}\label{ss:godel}
%%%%%%%%%%%%%%%%%%%%%%%%%%%%%%%%%%%%%%%%%%%%%%%%%%%%%%%%%%%%%%%%%%%
G\"odel's incompleteness' theorem has a striking version,
due to Chaitin, 1971-74 \cite{chaitin71, chaitin74},
in terms of Kolmogorov complexity.
In the language of arithmetic one can formalize partial
computability (this is G\"odel main technical ingredient
for the proof of the incompleteness theorem) hence also
Kolmogorov complexity.
\\
Chaitin proved an $n$ lower bound to the information content
of finite families of statements about finite restrictions
associated to an integer $n$ of the halting problem
or the values of $K$.
\\
In particular, for any formal system $\+T$, if $n$ is bigger
than the Kolmogorov complexity of $\+T$
(plus some constant, independent of $\+T$)
such statements cannot all be provable in $\+T$
\begin{theorem}[Chaitin, 1974 \cite{chaitin74}]
Suppose $\+O=\words$.
\\
1. Let $V:\words\to\+O$ be optimal (i.e. $K=K_V$).
Let $\+T_n$ be the family of true statements
$\exists p\ (V(p)=x)$ for $|x|\leq n$
(i.e. the halting problem for $V$ limited to the finitely
many words of length $\leq n$).
Then there exists a constant $c$ such that
$K(\+T_n) \geq n-c$ for all $n$.
\medskip\\
2. Let $\+T_n$ be the family of true statements
$K(x)\geq|x|)$ for $|x|\leq n$.
Then there exists a constant $c$ such that
$K(\+T_n) \geq n-c$ for all $n$.
\end{theorem}
\noindent
Note. In the statement of the theorem, $K(x)$ refers to
the Kolmogorov complexity on $\+O$ whereas $K(\+T_n)$
refers to that on an adequate elementary family
(cf. Definition \ref{def:elementary}).
%
%
%
%
%%%%%%%%%%%%%%%%%%%%%%%%%%%%%%%%%%%%%%%%%%%%%%%%%%%%%%%%%%%%%%%%%%%
%%%%%%%%%%%%%%%%%%%%%%%%%%%%%%%%%%%%%%%%%%%%%%%%%%%%%%%%%%%%%%%%%%%
%%%%%%%%%%%%%%%%%%%%%%%%%%%%%%%%%%%%%%%%%%%%%%%%%%%%%%%%%%%%%%%%%%%
\section{Formalization of randomness for finite objects}
%%%%%%%%%%%%%%%%%%%%%%%%%%%%%%%%%%%%%%%%%%%%%%%%%%%%%%%%%%%%%%%%%%%
%%%%%%%%%%%%%%%%%%%%%%%%%%%%%%%%%%%%%%%%%%%%%%%%%%%%%%%%%%%%%%%%%%%
%%%%%%%%%%%%%%%%%%%%%%%%%%%%%%%%%%%%%%%%%%%%%%%%%%%%%%%%%%%%%%%%%%%
%
%
%%%%%%%%%%%%%%%%%%%%%%%%%%%%%%%%%%%%%%%%%%%%%%%%%%%%%%%%%%%%%%%%%%%
\subsection{Probabilities: laws about a non formalized intuition}
%%%%%%%%%%%%%%%%%%%%%%%%%%%%%%%%%%%%%%%%%%%%%%%%%%%%%%%%%%%%%%%%%%%
%
Random objects {\em(words, integers, reals,...)}
constitute the basic intuition for probabilities
{\em... but they are not considered per se.}
No formal definition of random object is given:
there seems there is no need for such a formal concept.
The existing formal notion of {\em random variable} has nothing
to do with randomness: a random variable is merely a
{\em measurable function} which can be as non random as one likes.
\medskip\\
It sounds strange that the mathematical theory which deals with
randomness removes the natural basic questions:
\\\indent- {\em what is a random string?}
\\\indent- {\em what is a random infinite sequence?}
\\
When questioned, people in probability theory agree that they
skip these questions but do not feel sorry about it.
As it is, the theory deals with laws of randomness and
is so successful that it can do without entering this problem.
\medskip\\
This may seem to be analogous to what is the case in geometry.
What are points, lines, planes?
No definition is given, only relations between them.
Giving up the quest for an analysis of the nature of geometrical
objects  in profit of the axiomatic method
has been a considerable scientific step.
\\
However, we contest such an analogy.
Random objects are heavily used in many areas of science
and technology: sampling, cryptology,...
Of course, such objects are are in fact
{\em ``as much as we can random"}.
Which means {\em fake randomness}.
\begin{quote}{\em
Anyone who considers arithmetical methods of producing random reals
is, of course, in a state of sin.
For, as has been pointed out several times, there is no such thing
as a random number --- there are only methods to produce random
numbers, and a strict arithmetical procedure is of course not such
a method.}

\hfill{ John von Neumann, 1951 \cite{neumannsin}}
\end{quote}
So, what is ``true" randomness?
Is there something like a degree of randomness?
Presently, (fake) randomness only means to pass
some statistical tests.
One can ask for more.
\\
In fact, since Pierre Simon de Laplace (1749--1827),
some probabilists never gave up the idea of formalizing
the notion of random object.
Let's cite particularly Richard von Mises (1883--1953)
and Kolmogorov.
In fact, it is quite impressive that, having so brilliantly
and efficiently axiomatized probability theory via measure theory
in 1933 \cite{kolmo33}, Kolmogorov was not fully satisfied
of such foundations.\footnote{
Kolmogorov is one of the rare probabilists -- up to now --
not to believe that Kolmogorov's axioms for probability theory
do not constitute the last word about formalizing randomness...}
And kept a keen interest to the quest for a formal
notion of randomness initiated by von Mises in the 20's.
%
%
%%%%%%%%%%%%%%%%%%%%%%%%%%%%%%%%%%%%%%%%%%%%%%%%%%%%%%%%%%%%%%%%%%%
\subsection{The 100 heads paradoxical result in probability theory}
%%%%%%%%%%%%%%%%%%%%%%%%%%%%%%%%%%%%%%%%%%%%%%%%%%%%%%%%%%%%%%%%%%%
%
That probability theory fails to completely account for randomness
is strongly witnessed by the following paradoxical fact.
In probability theory,
{\em if we toss an unbiaised coin 100 times then
       100 heads are just as probable as any other outcome!}
Who really believes that ?
{\em\begin{quote}
The axioms of probability theory, as developped by Kolmogorov,
do not solve all mysteries that they are sometimes supposed to.

\hfill{Peter G\`acs \cite{gacs93}}
\end{quote}}

%Citer Durand et Zvonkin et la variante amusante avec les plaques 
%d'immatriculation
%
%%%%%%%%%%%%%%%%%%%%%%%%%%%%%%%%%%%%%%%%%%%%%%%%%%%%%%%%%%%%%%%%%%%%
%\subsection{Cryptology: fake but sufficient randomness}
%%%%%%%%%%%%%%%%%%%%%%%%%%%%%%%%%%%%%%%%%%%%%%%%%%%%%%%%%%%%%%%%%%%%
%%
%Contrarily to probability theory, cryptology heavily uses random
%objects.
%Though again, no formal definition is given, random sequences
%are produced which are in fact {\em ``as much as we can random"}.
%Which means {\bf\em fake randomness}.
%\begin{itemize}
%\item Fonction OneWay. L'existence en est un probl\`{e}me NP complet.
%\item M\^{e}mes algoritmmes, probablement (??\`{a}voir).
%\item Peut-on comparer : textes crypt\'{e}s et textes al\'{e}atoires. Si r\'{e}cursif 
%est remplac\'{e} par calcul\'{e} en temps p\^{o}lynomial, alors on aurait 
%l'al\'{e}atoire des cryptologues (??).
%\item Schnorr : "l'al\'{e}atoire, c'est le Oneway en temps polyn\^{o}mial".
%\item Comme dans la d\'{e}marche sous-jacente \`{a} la d\'{e}couverte de la 
%classification par compression, les diff\'{e}rents protagonistes de la 
%cryptologie sont sortis du domaine des puristes.
%\end{itemize}
%
%%%%%%%%%%%%%%%%%%%%%%%%%%%%%%%%%%%%%%%%%%%%%%%%%%%%%%%%%%%%%%%%%%%
\subsection{Kolmogorov's proposal: incompressible strings}
%%%%%%%%%%%%%%%%%%%%%%%%%%%%%%%%%%%%%%%%%%%%%%%%%%%%%%%%%%%%%%%%%%%
%
We now assume that $\+O=\words$, i.e. we restrict to words.
%
%------------------------------------------------------------------
\subsubsection{incompressibility with Kolmogorov complexity}
%------------------------------------------------------------------
Though much work has been devoted to get
{\em a mathematical theory of random objects},
notably by von Mises \cite{mises19,mises39},
none was satisfactory up to the 60's when Kolmogorov
based such a theory on Kolmogorov complexity,
hence on computability theory.
\\
The theory was, in fact, independently developed by
Gregory J. Chaitin (b. 1947), 1966
\cite{chaitin66}, 1969 \cite{chaitin69}
(both papers submitted in 1965).\footnote{
For a detailed analysis of {\em who did what, and when},
see \cite{livitanyi} p.89--92.}
\medskip\\
The basic idea is as follows:
{\em larger is the Kolmogorov complexity of a text,
more random is this text,
larger is its information content,
and more compressed is this text.}
\\
Thus, a theory for measuring the information content
is also a theory of randomness.
%Approches Top/Down et Bottom/Up de l'al\'{e}atoirit\'{e}.
%
\medskip\\
Recall that there exists $c$ such that for all $x\in\words$,
$K(x)\leq |x|+c$ (Proposition \ref{p:bound}).
Also, there is a ``stupid" program of length about $|x|$
which computes the word $x$ :
tell the successive letters of $x$.
The intuition of incompressibility is as follows:
$x$ is incompressible if ther no shorter way to get $x$.
\\
Of course, we are not going to define absolute randomness
for words. But a measure of randomness based on how far from
$|x|$ is $K(x)$.
\begin{definition}[Measure of incompressibility]$\\ $
A word $x$ is $c$-incompressible if $K(x)\geq|x|-c$.
\end{definition}
As is rather intuitive, most things are random.
The next Proposition formalizes this idea.
\begin{proposition}
The proportion of $c$-incompressible strings of length $n$
is $\geq 1-2^{-c}$.
\end{proposition}
{\small\begin{quote}
\begin{proof}
At most $2^{n-c}-1$ programs of length $<n-c$
and $2^n$ strings of length $n$.
\end{proof}
\end{quote}}
%
%
%------------------------------------------------------------------
\subsubsection{incompressibility with length conditional
Kolmogorov complexity}
%------------------------------------------------------------------
%
We observed in \S\ref{sss:entropy} that the entropy of a word
of the form $000...0$ is null.
I.e. entropy did not considered the information conveyed by
the length.
\\
Here, with incompressibility based on Kolmogorov complexity,
we can also ignore the information content conveyed by the length
by considering {\em incompressibility based on length conditional
Kolmogorov complexity}.
\begin{definition}[Measure of length conditional incompressibility]
A word $x$ is length conditional $c$-incompressible if
$K(x\mid |x|)\geq|x|-c$.
\end{definition}
The same simple counting argument yields the following Proposition.
\begin{proposition}
The proportion of length conditional $c$-incompressible
strings of length $n$ is $\geq 1-2^{-c}$.
\end{proposition}
A priori length conditional incompressibility is stronger
than mere incompressibility.
However, the two notions of incompressibility are about the
same \ldots up to a constant.
\begin{proposition}
There exists $d$ such that, for all $c\in\N$ and $x\in\words$
\medskip\\
1. $x$ is length conditional $c$-incompressible
$\Rightarrow$ $x$ is $(c+d)$-incompressible
\medskip\\
2. $x$ is $c$-incompressible $\Rightarrow$ $x$ is length conditional
$(2c+d)$-incompressible.
\end{proposition}
{\small\begin{quote}
\begin{proof}
1 is trivial.
For 2, observe that there exists $e$ such that, for all $x$,
$$
(*)\ \ \ \ K(x) \leq K(x\mid |x|) + 2 K(|x| - K(x\mid |x|)) + d
$$
In fact, if $K=K_\varphi$ and $K(\ \mid\ )=K_{\psi(\ \mid\ )}$
and
\medskip\\
$\begin{array}{rclcrcl}
|x| - K(x\mid |x|) &=& \varphi(p)
&& \psi(q\mid|x|) &=& x
\\
K(|x| - K(x\mid |x|)) &=& |p|
&& K(x\mid |x|) &=& |q|
\end{array}$
\medskip\\
With $p$ and $q$, hence with $\langle p,q\rangle$
(cf. Proposition \ref{p:code}), one can successively get
$\left\{\begin{array}{ll}
|x| - K(x\mid |x|) & \mbox{this is $\varphi(p)$}
\\
K(x\mid |x|) & \mbox{this is $q$}
\\
|x| & \mbox{just sum}
\\
x & \mbox{this is $\psi(q\mid|x|)$}
\end{array}\right.$
\\
Using $K\leq\log+c_1$ and $K(x)\geq|x|-c$, (*) yields
$$
|x|- K(x\mid |x|) \leq 2\log(|x| - K(x\mid |x|)) + 2c_1+c+d
$$
Finally, observe that $z\leq 2\log z +k$ insures $z\leq\max(8,2k)$.
\end{proof}
\end{quote}}
%
%
%%%%%%%%%%%%%%%%%%%%%%%%%%%%%%%%%%%%%%%%%%%%%%%%%%%%%%%%%%%%%%%%%%%
\subsection{Incompressibility is randomness: Martin-L\"{o}f's argument}
\label{ss:tests}
%%%%%%%%%%%%%%%%%%%%%%%%%%%%%%%%%%%%%%%%%%%%%%%%%%%%%%%%%%%%%%%%%%%
%
Now, if incompressibility is clearly a necessary condition for
randomness, how do we argue that it is a sufficient condition?
Contraposing the wanted implication, let's see that if a word
fails some statistical test then it is not incompressible.
We consider some spectacular failures of statistical tests.
\begin{example}\label{ex:ex1}
1. {\em [Constant left half length prefix]}
For all $n$ large enough, a string $0^nu$
with $|u|=n$ cannot be $c$-incompressible.
\medskip\\
2. {\em [Palindromes]}
Large enough palindromes cannot be
$c$-incompressible.
\medskip\\
3. {\em [$0$ and $1$ not equidistributed]}
For all $0<\alpha<1$, for all $n$ large enough,
a string of length $n$ which has $\leq \alpha\frac{n}{2}$ zeros
cannot be $c$-incompressible.
\end{example}
{\small\begin{quote}
\begin{proof}
1. Let $c'$ be such that $K(x)\leq|x|+c'$.
Observe that there exists $c''$ such that $K(0^nu) \leq K(u)+c''$
hence
$$
K(0^nu)\leq n+c'+c'' \leq \frac{1}{2}|0^nu|+c'+c''
$$
So that $K(0^nu)\geq|0^nu|-c$ is impossible for $n$ large enough.
\medskip\\
2. Same argument:
There exists $c''$ such that, for all palindrome $x$,
$$
K(x)\leq \frac{1}{2}|x| + c''
$$
\medskip\\
3. The proof follows the classical argument to get the law
of large numbers (cf. Feller's book \cite{feller}).
Let's do it for $\alpha=\frac{2}{3}$, so that
$\frac{\alpha}{2}=\frac{1}{3}$.
\medskip\\
Let $A_n$ be the set of strings of length $n$ with
$\leq \frac{n}{3}$ zeros.
We estimate the number $N$ of elements of $A_n$.
$$
N=\sum_{i=0}^{i=\frac{n}{3}} \binom{n}{i}
=\sum_{i=0}^{i=\frac{n}{3}} \frac{n!}{i!\ (n-i)!}
\leq (\frac{n}{3}+1)\ \frac{n!}{\frac{n}{3}!\ \frac{2n}{3}!}
$$
Use Stirling's formula (1730)
$$
\sqrt{2n\pi}\ {\left(\frac{n}{e}\right)}^n\ e^{\frac{1}{12n+1}}
< n!
< \sqrt{2n\pi}\ {\left(\frac{n}{e}\right)}^n\ e^{\frac{1}{12n}}
$$
$$
N < n \frac{\sqrt{2n\pi}\ {\left(\frac{n}{e}\right)}^n}
{\sqrt{2\frac{n}{3}\pi}\
{\left(\frac{\frac{n}{3}}{e}\right)}^{\frac{n}{3}}\
\sqrt{2\frac{2n}{3}\pi}\
{\left(\frac{\frac{2n}{3}}{e}\right)}^{\frac{2n}{3}}}\
= \frac{3}{2}\
  \sqrt{\frac{n}{\pi}}\ {\left(\frac{3}{\sqrt[3]{4}}\right)}^n
$$
Using Proposition \ref{p:rank}, for any element of $A_n$, we have
$$
K(x\mid n)
\leq \log(N) +d
\leq n\log\left(\frac{3}{\sqrt[3]{4}}\right) + \frac{\log n}{2} + d
$$
Since $\frac{27}{4} < 8$, we have $\frac{3}{\sqrt[3]{4}} < 2$
and $\log\left(\frac{3}{\sqrt[3]{4}}\right) < 1$.
Hence,
$n-c \leq n\log\left(\frac{3}{\sqrt[3]{4}}\right)+\frac{\log n}{2}+d$
is impossible for $n$ large enough.
\\
So that $x$ cannot be $c$-incompressible.
\end{proof}
\end{quote}}
Let's give a common framework to the three above examples
so as to get some flavor of what can be a statistical test.
To do this, we follow the above proofs of compressibility.
\begin{example}\label{ex:ex2}
1. {\em [Constant left half length prefix]}\\
Set $V_m = \mbox{ all strings with $m$ zeros ahead}$.
The sequence $V_0,V_1,...$ is decreasing.
The number of strings of length $n$ in $V_m$ is $0$ if $m>n$
and $2^{n-m}$ if $m\leq n$.
Thus, the proportion
$\frac{\sharp\{x \mid |x|=n\ \wedge\ x\in V_m\}}{2^n}$
of length $n$ words which are in $V_m$ is $2^{-m}$.
\medskip\\
2. {\em [Palindromes]}
Put in $V_m$ all strings which have equal length $m$
prefix and suffix.
The sequence $V_0,V_1,...$ is decreasing.
The number of strings of length $n$ in $V_m$ is
$0$ if $m>\frac{n}{2}$
and $2^{n-2m}$ if $m\leq \frac{n}{2}$.
Thus, the proportion of length $n$ words which are in $V_m$
is $2^{-2m}$.
\medskip\\
3. {\em [$0$ and $1$ not equidistributed]}
Put in $V^\alpha_m =$ all strings $x$ such that
the number of zeros is
$\leq (\alpha+(1-\alpha)2^{-m})\frac{|x|}{2}$.
The sequence $V_0,V_1,...$ is decreasing.
A computation analogous to that done in the proof of the law
of large numbers shows that the proportion of length $n$ words
which are in $V_m$ is $\leq 2^{-\gamma m}$ for some $\gamma>0$
(independent of $m$).
\end{example}
Now, what about other statistical tests?
But what is a statistical test?
A convincing formalization has been developed by Martin-L\"of.
The intuition is that illustrated in Example \ref{ex:ex2}
augmented of the following feature:
each $V_m$ is computably enumerable and so is the relation
$\{(m,x) \mid x\in V_m\}$.
A feature which is analogous to the partial computability
assumption in the definition of Kolmogorov complexity.
\begin{definition}\label{def:test}
[Abstract notion of statistical test, Martin-L\"of, 1964]
A statistical test is a family of nested critical regions
$$
\{0,1\}^*\supseteq V_0\supseteq V_1\supseteq V_2\supseteq
...\supseteq V_m\supseteq...
$$
such that $\{(m,x) \mid x\in V_m\}$ is computably enumerable
and the proportion
$\frac{\sharp\{x \mid |x|=n\ \wedge\ x\in V_m\}}{2^n}$
of length $n$ words which are in $V_m$ is $2^{-m}$.
\medskip\\
Intuition. The bound $2^{-m}$ is just a normalization.
Any bound $b(n)$ such that $b:\N\to\Q$ which is computable, 
decreasing and with limit $0$ could replace $2^{-m}$. 
\\
The significance of $x\in V_m$ is that the hypothesis
{\em $x$ is random} is rejected with significance level $2^{-m}$.
\end{definition}
\begin{remark}
Instead of sets $V_m$ one can consider a function
$\delta:\words\to\N$ such that
$\frac{\sharp\{x \mid |x|=n\ \wedge\ x\in V_m\}}{2^n}\leq 2^{-m}$
and $\delta$ is computable from below, i.e.
$\{(m,x) \mid \delta(x)\geq m\}$ is recursively enumerable.
\end{remark}
We have just argued on some examples that all statistical tests
from practice are of the form stated by Definition \ref{def:test}.
Now comes Martin-L\"of fundamental result about statistical tests
which is in the vein of the invariance theorem.
\begin{theorem} [Martin-L\"of, 1965]\label{thm:test}
Up to a constant shift, there exists a largest statistical test
$(U_m)_{m\in\N}$
$$
\forall (V_m)_{m\in\N}\ \exists c\ \forall m\ \
V_{m+c} \subseteq U_m
$$
In terms of functions, up to an additive constant,
there exists a largest statistical test $\Delta$
$$
\forall \delta\ \exists c\ \forall x\ \
\delta(x)<\Delta(x)+c
$$
\end{theorem}
{\small\begin{quote}
\begin{proof}
Consider $\Delta(x)=|x|-K(x\mid|x|)-1$.
\\
\fbox{\em $\Delta$ is a test.}
Clearly, $\{(m,x) \mid \Delta(x)\geq m\}$ is computably enumerable.
\\
$\Delta(x)\geq m$ means $K(x\mid|x|)\leq |x|-m-1$.
So no more elements in $\{x \mid \Delta(x)\geq m\ \wedge\ |x|=n\}$
than programs of length $\leq n-m-1$, which is $2^{n-m}-1$.
\\
\fbox{\em $\Delta$ is largest.}
$x$ is determined by its rank in the set
$V_{\delta(x)}=\{z \mid \delta(z)\geq \delta(x)\ \wedge\ |z|=|x|\}$.
Since this set has $\leq 2^{n-\delta(x)}$ elements,
the rank of $x$ has a binary representation of length
$\leq |x|-\delta(x)$.
Add useless zeros ahead to get a word $p$ with length
$|x|-\delta(x)$.
\\
With $p$ we get $|x|-\delta(x)$.
With $|x|-\delta(x)$ and $|x|$ we get $\delta(x)$ and construct
$V_{\delta(x)}$. With $p$ we get the rank of $x$ in this set,
hence we get $x$.
Thus,\\
$K(x\mid|x|) \leq |x|-\delta(x) +c$,
i.e. $\delta(x)<\Delta(x)+c$.
\end{proof}
\end{quote}}
The importance of the previous result is the following corollary
which insures that, for words,
incompressibility implies (hence is equivalent to) randomness.
\begin{corollary}[Martin-L\"of, 1965]
Incompressibility passes all statistical tests.
I.e. for all $c$, for all statistical test $(V_m)_m$,
there exists $d$ such that
$$
\forall x\ (x\mbox{ is $c$-incompressible }
\Rightarrow\ x\notin V_{c+d})
$$
\end{corollary}
{\small\begin{quote}
\begin{proof}
Let $x$ be length conditional $c$-incompressible.
This means that $K(x\mid |x|)\geq |x|-c$.
Hence $\Delta(x)=|x|-K(x\mid|x|)-1\leq c-1$, which means that
$x\notin U_c$.
\\
Let now $(V_m)_m$ be a statistical test.
Then there is some $d$ such that $V_{m+d}\subseteq U_m$
Therefore $x\notin V_{c+d}$.
\end{proof}
\end{quote}}
\begin{remark}
Observe that incompressibility is a {\em bottom-up} notion:
we look at the value of $K(x)$ (or that of $K(x \mid |x|)$).
\\
On the opposite, passing statistical tests is a {\em top-down}
notion.
To pass all statistical tests amounts to an inclusion
in an intersection: namely, an inclusion in
$$
\bigcap_{(V_m)_m}\ \bigcup_c\ V_{m+c}
$$
\end{remark}
%
%
%%%%%%%%%%%%%%%%%%%%%%%%%%%%%%%%%%%%%%%%%%%%%%%%%%%%%%%%%%%%%%%%%%%
\subsection{Randomness: a new foundation for probability theory?}
%%%%%%%%%%%%%%%%%%%%%%%%%%%%%%%%%%%%%%%%%%%%%%%%%%%%%%%%%%%%%%%%%%%
%
Now that there is a sound mathematical notion of randomness
(for finite objects), or more exactly a measure of randomness,
is it possible/reasonable to use it as a new foundation for
probability theory?
\\
Kolmogorov has been ambiguous on this question.
In his first paper on the subject (1965, \cite{kolmo65}, p. 7),
Kolmogorov briefly evoked that possibility :
\begin{quote}
\dots to consider the use of the
[Algorithmic Information Theory] constructions
in providing a new basis for Probability Theory.
\end{quote}
However, later (1983, \cite{kolmo83}, p. 35--36),
he separated both topics
\begin{quote}
``there is no need whatsoever to change the
established construction of the mathematical
probability theory on the basis of the general theory
of measure.
I am not enclined to attribute the significance of
necessary foundations of probability theory to the
investigations [about Kolmogorov complexity] that
I am now going to survey.
But they are most interesting in themselves.
\end{quote}
though stressing the role of his new theory of random
objects for {\em mathematics as a whole}
(\cite{kolmo83}, p. 39):
\begin{quote}
The concepts of information theory as applied
to infinite sequences give rise to very interesting
investigations, which, without being indispensable
as a basis of probability theory, can acquire a
certain value in the investigation of the
algorithmic side of mathematics as a whole.
\end{quote}
%
%
%
%
%
%
%%%%%%%%%%%%%%%%%%%%%%%%%%%%%%%%%%%%%%%%%%%%%%%%%%%%%%%%%%%%%%%%%%%
%%%%%%%%%%%%%%%%%%%%%%%%%%%%%%%%%%%%%%%%%%%%%%%%%%%%%%%%%%%%%%%%%%%
%%%%%%%%%%%%%%%%%%%%%%%%%%%%%%%%%%%%%%%%%%%%%%%%%%%%%%%%%%%%%%%%%%%
\section{Formalization of randomness for infinite objects}
%%%%%%%%%%%%%%%%%%%%%%%%%%%%%%%%%%%%%%%%%%%%%%%%%%%%%%%%%%%%%%%%%%%
%%%%%%%%%%%%%%%%%%%%%%%%%%%%%%%%%%%%%%%%%%%%%%%%%%%%%%%%%%%%%%%%%%%
%%%%%%%%%%%%%%%%%%%%%%%%%%%%%%%%%%%%%%%%%%%%%%%%%%%%%%%%%%%%%%%%%%%
%
We shall stick to infinite sequences of zeros and ones:
$\{0,1\}^\N$.
%
%
%%%%%%%%%%%%%%%%%%%%%%%%%%%%%%%%%%%%%%%%%%%%%%%%%%%%%%%%%%%%%%%%%%%
\subsection{Martin-L\"of approach with topology and computability}
%%%%%%%%%%%%%%%%%%%%%%%%%%%%%%%%%%%%%%%%%%%%%%%%%%%%%%%%%%%%%%%%%%%
%
This approach is an extension to infinite sequences of the one
he developed for finite objects, cf. \S\ref{ss:tests}.
\medskip\\
To prove a probability law amounts to prove that
a certain set $X$ of sequences has probability one.
To do this, one has to prove that the complement set
$Y=\cantor \setminus X$ has probability zero.
Now, in order to prove that $Y \subseteq \cantor$ has
probability zero, basic measure theory tells us that
one has to include $Y$ in open sets with
arbitrarily small probability.
I.e. for each $n\in \N$ one must find an open set
$U_{n}\supseteq Y$
which has probability $\leq \frac{1}{2^n}$.
\\
If things were on the real line ${\bf R}$
we would say that $U_{n}$ is a countable union of
intervals with rational endpoints.
\\
Here, in $\cantor$, $U_{n}$ is a
countable union of sets of the form
$u\cantor$ where $u$ is a finite binary
string and $u\cantor$ is the set of infinite sequences
which extend $u$.
\\
In order to prove that $Y$ has probability zero,
for each $n\in \N$ one must find a family
$(u_{n,m})_{m\in\N}$ such that
$Y\subseteq \bigcup_{m} u_{n,m}\cantor$
and $Proba(\bigcup_{m} u_{n,m}\cantor)\leq \frac{1}{2^n}$
for each $n\in \N$.
\\
Now, Martin-L\"of makes a crucial observation:
mathematical probability laws which we can consider
necessarily have some effective character.
And this effectiveness should reflect in the proof
as follows:
{\em the doubly indexed sequence
$(u_{n,m})_{{n,m\in\N}}$ is computable.}

Thus, the set $\bigcup_{m} u_{n,m}\cantor$ is a
{\em computably enumerable open set} and
$\bigcap_{n} \bigcup_{m} u_{n,m}\cantor$
is a countable intersection of a
{\em computably enumerable family of open sets}.
\medskip\\
Now comes the essential theorem, which is completely analog
to Theorem \ref{thm:test}.
\begin{theorem}[Martin-L\"of \cite{martinlof66}]
Let's call constructively null $G_\delta$ set any set of the form
$\bigcap_{n} \bigcup_{m} u_{n,m}\cantor$
where the sequence $u_{n,m}$ is computably enumerable
and $Proba(\bigcup_{m} u_{n,m}\cantor)\leq \frac{1}{2^n}$
(which implies that the intersection set has probability zero).
\\
There exist a largest constructively null $G_\delta$ set
\end{theorem}
Let's insist that the theorem says {\em largest}, up to nothing, 
really largest.
\begin{definition}[Martin-L\"of \cite{martinlof66}]
A sequence $\alpha\in\cantor$ is random if it belongs to no
constructively null $G_\delta$ set
(i.e. if it does not belongs to the largest one).
\end{definition}
In particular, the family of random sequence, being the complement
of a constructively null $G_\delta$ set, has probability $1$.
%
%
%
%%%%%%%%%%%%%%%%%%%%%%%%%%%%%%%%%%%%%%%%%%%%%%%%%%%%%%%%%%%%%%%%%%%
\subsection{The bottom-up approach}
%%%%%%%%%%%%%%%%%%%%%%%%%%%%%%%%%%%%%%%%%%%%%%%%%%%%%%%%%%%%%%%%%%%
%
%------------------------------------------------------------------
\subsubsection{The naive idea badly fails}
%------------------------------------------------------------------
The natural naive idea is to extend randomness from finite objects
to infinite ones.
The obvious first approach is to consider sequences
$\alpha\in\cantor$ such that, for some $c$,
\begin{equation}
\forall n\ \ K(\alpha\segment n)\geq n-c
\end{equation}
However, Martin-L\"of proved that there is no such sequence.
\begin{theorem}[Martin-L\"of \cite{martinlof71}]\label{thm:no}
For every $\alpha\in\cantor$ there are infinitely many $k$
such that $K(\alpha\segment k) \leq k-\log k -O(1)$.
\end{theorem}
{\small\begin{quote}
\begin{proof}
Let $f(z)= k -(\lfloor\log z\rfloor+1)$
First, observe that
$$
f(z+2)-f(z)
= 2 - \lfloor\log (z+2)\rfloor + \lfloor\log z\rfloor
= 2 - (\lfloor\log z + \log(1+\frac{2}{z}\rfloor
                                  - \lfloor\log z\rfloor)
>0
$$
since $\log(1+\frac{2}{z}\leq 1$ for $kz\geq1$.
\\
Fix any $m$ and consider $\alpha\segment m$.
This word is the binary representation of some integer $k$
such that $m=\lfloor\log k\rfloor +1$.
Now, consider $x=\alpha\segment k$ and let $y$ be the suffix
of $x$ of length $k-m=f(k)$.
From $y$ we get $|y|=k-m=f(k)$.
Since $f(z+2)-f(z)>0$, there are at most two (consecutive)
integers $k$ such that $f(k)=|y|$.
One bit of information tells which one in case there are two
of them.
So, from $y$ (plus one bit of information) one gets $m$.
Hence the binary representation of $m$,
which is $\alpha\segment m$.
By concatenation with $y$, we recover $x=\alpha\segment k$.
\\
This process being effective, Proposition \ref{p:bound}
(point 3) insures that
$$
K(\alpha\segment k) \leq K(y)+O(1) \leq |y|+O(1)
= k-m +O(1) = k-\log k +O(1)
$$
\end{proof}
\end{quote}}
The above argument can be extended to prove a much more general
result.
\begin{theorem}[Large oscillations,
Martin-L\"of, 1971 \cite{martinlof71}]
\label{thm:osc}
Let $f:\N\to\N$ be a total computable function satisfying
$\sum_{n\in\N}2^{-g(n)}=+\infty$.
Then, for every $\alpha\in\cantor$, there are infinitely many $k$
such that $K(\alpha\segment k) \leq k-f(k)$.
\end{theorem}
%
%
%
%------------------------------------------------------------------
\subsubsection{Miller \& Yu's theorem}
%------------------------------------------------------------------
%
It took about forty years to get a characterization of randomness
via plain Kolmogorov complexity which completes very simply
Theorem \ref{thm:osc}.
\begin{theorem}[Miller \& Yu, 2004 \cite{milleryu}]
\label{thm:milleryu}
1. Let $f:\N\to\N$ be a total computable function satisfying
$\sum_{n\in\N}2^{-g(n)}<+\infty$.
Then, for every random $\alpha\in\cantor$,
there exists $c$ such that
$K(\alpha\segment k \mid k) \geq k-f(k)-c$ for all $k$.
\medskip
2. There exists a total computable function $f:\N\to\N$ satisfying
$\sum_{n\in\N}2^{-g(n)}<+\infty$ such that
for every non random $\alpha\in\cantor$
there are infinitely many $k$ such that
$K(\alpha\segment k) \leq k-f(k)$.
\end{theorem}
Recently, an elementary proof of this theorem
was given by Bienvenu \& Merkle \& Shen, \cite{bms07}.
%
%
%
%------------------------------------------------------------------
\subsubsection{Kolmogorov randomness and $\emptyset'$}
\label{sss:nies}
%------------------------------------------------------------------
A natural question following Theorem \ref{thm:no} is to look at
the so-called Kolmogorov random sequences which satisfy
$K(\alpha\segment k)\geq k-O(1)$
for infinitely many $k$'s.
This question got a very surprising answer involving randomness
with oracle the halting problem $\emptyset'$.
\begin{theorem}[Nies, Stephan \& Terwijn \cite{nies}]
\label{thm:nies}
Let $\alpha\in\cantor$.
There are infinitely many $k$
such that $K(\alpha\segment k) \leq k-f(k)$
(i.e. $\alpha$ is Kolmogorov random)
if and only if $\alpha$ is $\emptyset'$-random.
\end{theorem}
%
%
%
%------------------------------------------------------------------
\subsubsection{Variants of Kolmogorov complexity and randomness}
%------------------------------------------------------------------
Bottom-up characterization of random sequences were
obtained by Chaitin, Levin  and Schnorr using diverse
variants of Kolmogorov complexity..
\begin{definition}\label{def:SH}
1. [Schnorr, \cite{schnorr71} 1971]
For $\+O=\words$, the process complexity $S$ is the variant of
Kolmogorov complexity obtained by restricting to partial computable
functions $\words\to\words$ which are monotonous,
i.e. if $p$ is a prefix of $q$ and $V(p), V(q)$ are both defined
then $V(p)$ is a prefix of $V(q)$.
\medskip\\
2. [Chaitin, \cite{chaitin75} 1975]
The prefix-free variant $H$ of Kolmogorov
complexity is obtained by restricting to partial computable
functions $\words\to\words$ which have prefix-free domains.
\medskip\\
3. [Levin, \cite{zvonkin-levin} 1970]
For $\+O=\words$, the monotone variant $Km$ of Kolmogorov
complexity is obtained as follows:
$Km(x)$ is the least $|p|$ such that $x$ is a prefix of $U(p)$
where $U$ is universal among monotone partial computable functions.
\end{definition}
\begin{theorem}\label{thm:HS}
Let $\alpha\in\cantor$. The following conditions
Then $\alpha$ is random
if and only if $S(\alpha\segment k) \geq k -O(1)$
if and only if $H(\alpha\segment k) \geq k -O(1)$
if and only if $Km(\alpha\segment k) \geq k -O(1)$.
\end{theorem}
The main problem with these variants of Kolmogorov complexity
is that {\em there is no solid understanding of what the restrictions
they involve really mean}.
\\
Chaitin has introduced the idea of self-delimitation for prefix-free
functions: since a program in the domain of $U$ has no extension
in the domain of $U$, it somehow know where it ends.
Though interesting, this interpretation is not a definitive
explanation as Chaitin himself admits (personal communication).
\\
Nevertheless, these variants have wonderful properties.
Let's cite one of the most striking one:
taking $\+O=\N$, the series $2^{-H(n)}$ converges and is the biggest
absolutely convergent series up to a multiplicative factor.
%
%
%
%
%
%
%
%%%%%%%%%%%%%%%%%%%%%%%%%%%%%%%%%%%%%%%%%%%%%%%%%%%%%%%%%%%%%%%%%%%
%%%%%%%%%%%%%%%%%%%%%%%%%%%%%%%%%%%%%%%%%%%%%%%%%%%%%%%%%%%%%%%%%%%
%%%%%%%%%%%%%%%%%%%%%%%%%%%%%%%%%%%%%%%%%%%%%%%%%%%%%%%%%%%%%%%%%%%
\section{Application of Kolmogorov complexity to classification}
%%%%%%%%%%%%%%%%%%%%%%%%%%%%%%%%%%%%%%%%%%%%%%%%%%%%%%%%%%%%%%%%%%%
%%%%%%%%%%%%%%%%%%%%%%%%%%%%%%%%%%%%%%%%%%%%%%%%%%%%%%%%%%%%%%%%%%%
%%%%%%%%%%%%%%%%%%%%%%%%%%%%%%%%%%%%%%%%%%%%%%%%%%%%%%%%%%%%%%%%%%%
%
%
%%%%%%%%%%%%%%%%%%%%%%%%%%%%%%%%%%%%%%%%%%%%%%%%%%%%%%%%%%%%%%%%%%%
\subsection{What is the problem?}\label{ss:what}
%%%%%%%%%%%%%%%%%%%%%%%%%%%%%%%%%%%%%%%%%%%%%%%%%%%%%%%%%%%%%%%%%%%
%
Striking results have been obtained, using Kolmogorov complexity,
with the problem of classifying quite diverse families of objects:
let them be literary texts, music pieces,
examination scripts (lax supervised) or, at a different level,
natural languages,
natural species (philogeny).
\\
The authors, mainly Bennett, Vitanyi, Cilibrasi,..
have worked out refined methods which are along the following lines.
\begin{enumerate}
\item[(1)]
Define a specific family of objects which we want to classify.
\\
For example a set of Russian literary texts that we want to
group by authors. In this simple case all texts are written
in their original Russian language.
Another instance, music. In that case a common translation is
necessary, i.e. a normalization of the texts of these music pieces
that we want to group by composer. This is required
in order to be able to compare them.
An instance at a different level: the 52 main european languages.
In that case one has to choose a canonical text and its
representations in each one of the different languages
(i.e. corpus) that we consider.
For instance, the Universal Declaration of Human Rights
and its translations in these languages,
an example which was a basic test for Vitanyi's method.
As concerns natural species, the canonical object will be a
DNA sequence.
\\
What has to be done is to select, define and normalize a family
of objects or corpus that we want to classify.
\\
Observe that this is not always an obvious step:
    \begin{itemize}
    \item
    There may be no possible normalization. For instance with
    artists paintings,.
    \item
    The family to be classified may be finite though ill defined
    or even of unknown size, cf. \ref{sss:ngd}.
    \end{itemize}
\item[(2)]
In fine we are with a family of words on some fixed alphabet
representing objects for which we want to compare and measure
pairwise the common information content.
\\
This is done by defining a distance for these pairs
of (binary) words with the following intuition:
\begin{quote}
the more common information there is between two words,
the closer they are and the shorter is their distance.
Conversely, the less common information there is between two words,
the more they are independent and non correlated,
and bigger is their distance.
\\
Two identical words have a null distance.
Two totally independent words
(for example, words representing $100$ coin tossing)
have distance about $1$ (for a normalized distance bounded by $1$).
\end{quote}
Observe that the authors follow Kolmogorov's approach
which was to define a numerical measure of information content of words, i.e. a measure of their randomness.
In exactly the same way, a volume or a surface gets a numerical
measure.
\item[(3)]
Associate a classification to the objects or corpus defined in
(1) using the numerical measures of the distances introduced in (2).
\\
This step is presently the least formally defined.
The authors give representations of the obtained classifications
using tables, trees, graphs,...
This is indeed more a visualization of the obtained classification
than a formal classification.
Here the authors have no powerful formal framework such as,
for example, Codd's relational model of data bases and its extension
to object data bases with trees.
How are we to interpret their tables or trees?
We face a problem, a classical one.
for instance with distances between DNA sequences,
Or with the acyclic graph structure of Unix files in a computer.
\\
This is much as with the rudimentary, not too formal, classification
of words in a dictionary of synonyms.
\\
Nevertheless, Vitanyi \& al. obtained by his methods
a classification tree for the 52 European languages which is that
obtained by linguists, a remarkable success.
And the phylogenetic trees relative to parenthood which are precisely
those obtained via DNA sequence comparisons by biologists.
\item[(4)]
An important problem remains to use a distance to obtain
a classification as in (3). Let's cite Cilibrasi \cite{cilibrasi03}:
\begin{quote}
Large objects (in the sense of long strings) that differ
by a tiny part are intuitively closer than tiny objects that
differ by the same amount.
For example, two whole mitochondrial genomes of $18,000$ bases
that differ by $9,000$ are very different, while two whole nuclear
genomes of $3\times 10^9$ bases that differ by only $9,000$ bases
are very similar.
Thus, absolute difference between two objects does not govern
similarity, but relative difference seems to.
\end{quote}
As we shall see, this problem is easy to fix by some normalization
of distances.
\item[(5)]
Finally, all these methods rely on Kolmogorov complexity which
is a non computable function (cf. \S\ref{ss:Knoncomput}).
The remarkable idea introduced by Vitanyi is as follows:
\begin{itemize}
\item
consider the Kolmogorov complexity of an object
as the ultimate and ideal value of the compression of that object,
\item
and compute approximations of this ideal compression
using usual efficient compressors such as gzip, bzip2, PPM,...
\end{itemize}
Observe that the quality and fastness of such compressors
is largely due to heavy use of statistical tools.
For example, PPM (Prediction by Partial Matching) uses a pleasing
mix of statistical models arranged by trees, suffix trees
or suffix arrays.
The remarkable efficiency of these tools is of course due to
several dozens of years of research in data compression.
And as time goes on, they improve and better approximate
Kolmogorov complexity.
\\
Replacing the ``pure' but non computable Kolmogorov complexity
by a banal compression algorithm such as gzip
is quite a daring step took by Vitanyi!
\\
\end{enumerate}
%
%
%%%%%%%%%%%%%%%%%%%%%%%%%%%%%%%%%%%%%%%%%%%%%%%%%%%%%%%%%%%%%%%%%%%
\subsection{Classification via compression}
%%%%%%%%%%%%%%%%%%%%%%%%%%%%%%%%%%%%%%%%%%%%%%%%%%%%%%%%%%%%%%%%%%%
%
%------------------------------------------------------------------
\subsubsection{The normalized information distance $NID$}
\label{ss:nid}
%------------------------------------------------------------------
We now formalize the notions described above.
The idea is to measure the information content shared
by two binary words representing some objects in a family
we want to classify.
\\
The first such tentative goes back to the 90's \cite{bennett}:
Bennett and al. define a notion of {\em information distance}
between two words $x,y$ as the size of the shortest program
which maps $x$ to $y$ and $y$ to $x$.
These considerations rely on the notion of reversible computation.
A possible formal definition for such a distance is
$$
ID(x,y) = \mbox{least $|p|$ such that $U(p,x)=y$ and $U(p,y)=x$}
$$
where $U:\words\times\words\to\words$ is optimal for $K(\ |\ )$.
\medskip\\ 
An alternative definition is as follows: s
$$
ID'(x,y) = \max\{K(x|y),K(y|x)\}
$$
The intuition for these definitions is that the shortest program
which computes $x$ from $y$ takes into accoulnt all similarities
between $x$ and $y$.
\\
Observe that the two definitions do not coincide
(even up to logarithmic terms) but lead to similar
developments and efficient applications.
\begin{note*}
In the above formula, $K$ can be plain Kolmogorov complexity
or its prefix version.
In fact, this does not matter for a simple reason:
all properties involving this distance will be true up to a
$O(\log(|x|),\log(|y|))$ term and the difference between
$K(z|t)$ and $H(z|t)$ is bounded by $2\log(|z|)$.
For conceptual simplicity, we stick to plain Kolmogorov complexity.
\end{note*}
$ID$ and $ID'$ satisfy the axioms of a distance
{\em up to a logarithmic term}.
The strict axioms for a distance $d$ are
\medskip\\\medskip
$\left\{\begin{array}{rcll}
d(x,x)&=&0&\mbox{(identity)}\\
d(x,y)&=&d(y,x)&\mbox{(symmetry)}\\
d(x,z)&\leq&d(x,y)+d(y,z)&\mbox{(triangle inequality)}
\end{array}\right.$
\medskip\\
The up to a $\log$ term axioms which are satisfied by $ID$
and $ID'$ are as follows:
\medskip\\\medskip
$\left\{\begin{array}{rcl}
ID(x,x)&=&O(1)\\
ID(x,y)&=&ID(y,x)\\
ID(x,z)&\leq&ID(x,y)+ID(y,z) + O(\log(ID(x,y)+ID(y,z)))
\end{array}\right.$
{\small\begin{quote}
\begin{proof}
Let $e$ be such that $U(e,x)=x$ for all $x$.
Then $ID(x,x)\leq|e|=O(1)$. No better upper bound is possible
(except if we assume that the empty word is such an $e$).
\\
Let now $p,p',q,q'$ be shortest programs such that
$U(p,y)=x$, $U(p',x)=y$, $U(q,z)=y$, $U(q',y)=z$.
Thus,
$K(x|y)=|p|$, $K(y|x)=|p'|$, $K(y|z)=|q|$, $K(z|y)=|q'|$.
\\
Consider the injective computable function $\langle\ \rangle$
of Proposition~\ref{p:codeloglog} which is such that
$|\langle r,s \rangle|=|r|+|s|+O(\log |r|)$.
\\
Set $\varphi:\words\times\words\to\words$ be such that
$\varphi(\langle r,s \rangle, x) = U(s,U(r,x))$.
Then
$$
\varphi(\langle q,p \rangle,z) = U(p,U(q,z)) = U(p,y) = x
$$
so that, by the invariance theorem,
\begin{multline*}
K(x|z)\leq K_\varphi(x|z)+O(1)\leq|\langle q,p \rangle|+O(1)
\\
= |q|+|p|+ O(\log(|q|))
= K(y|z)+K(x|y)+O(\log(K(y|z)))
\end{multline*}
And similarly for the other terms.
Which proves the stated approximations of the axioms.
\end{proof}
\end{quote}}
It turns out that such approximations of the axioms
are enough for the development of the theory.
%In particular, this distance is smallest, up to a constant,
%among the so-called admissible distances
%
\medskip\\
As said in \S\ref{ss:what}, to avoid scale distortion,
this distance $ID$ is normalized to $NID$
(normalized information distance) as follows:
$$
NID(x,y) = \frac{\max(K(x|y),K(y|x))}{\max(K(x),K(y))}
$$
The remaining problem is that this distance is not computable
since $K$ is not.
Here comes Vitanyi's daring idea:
consider this $NID$ as an ideal distance which
is to be approximated by replacing the Kolmogorov function $K$
by computable compression algorithms which go on improving. 
%
%
%------------------------------------------------------------------
\subsubsection{The normalized compression distance $NCD$}
%------------------------------------------------------------------
%Axiomes des compresseurs normaux page 29 chez Rudi
%
The approximation of $K(x)$ by $C(x)$ where $C$ is a compressor,
does not suffice. We also have to approximate the conditional
Kolmogorov complexity $K(x|y)$.
Vitanyi chooses the following approximation:
$$
C(y|x) = C(xy) - C(x)
$$
The authors explain as follows their intuition.
\\ To compress the word $xy$ ($x$ concatenated to $y$),
\\- the compressor first compresses $x$,
\\- then it compresses $y$ but skip all information from $y$
which was already in $x$.
\\ Thus, the output is not a compression of $y$ but a compression
of $y$ with all $x$ information removed.
I.e. this output is a {\em conditional compression}ñ of $y$ knowing $x$.
\\
Now, the assumption that the compressor first compresses $x$
is questionable: how does the compressor recovers $x$ in $xy$ ?.
One can argue positively in case $x,y$ are random
(= incompressible) and in case $x=y$.
And between these two extreme cases? But it works...
The miracle of modelization?
Or something not completely understood?
%
%When compressing $xy$, the information common to $x$ and $y$
%is compressed only once. So that $C(xy)$ counts
%the information specific to $x$ and that specific to $y$
%plus that common to $x$ and $y$
%whereas $C(x)+C(y)$ counts the information specific to $x$ and
%that specific to $y$ {\em plus twice} that common to $x$ and $y$.
%In this way, the difference $C(xy)-C(x)$ is the information
%
\medskip\\
With this approximation, plus the assumption that
$C(xy)=C(yx)$ (also questionable) we get the following approximation
of $NID$, called the normalized compression distance $NCD$ :
\begin{eqnarray*}
NCD(x,y) &=& \frac{\max(C(x|y),C(y|x))}{\max(C(x),C(y))}\\
&=& \frac{\max(C(yx)-C(y),C(xy)-C(x))}{\max(C(x),C(y))}\\
&=& \frac{C(xy)-\min(C(x),C(y))}{\max(C(x),C(y))}
\end{eqnarray*}
Clustering according to $NCD$ and, more generally, classification
via compression, is a kind of black box: words are grouped
together according to features that are not explicitly known to us.
Moreover, there is no reasonable hope that the analysis of the
computation done by the compressor gives some light on the obtained
clusters.
For example, what makes a text by Tolsto\"{\i} so characteristic?
What differentiates the styles of Tolsto\"{\i} and Dostoievski?
But it works, Russian texts are grouped by authors by a compressor
which ignores everything about Russian literature.
\\
When dealing with some classification obtained by compression,
one should have some idea of this classification:
this is semantics whereas the compressor is purely syntactic
and does not understand anything.
\\
This is very much like with machines which, given some formal
deduction system, are able to prove quite complex statements.
But these theorems are proved with no explicit semantical idea,
how are we to interpret them? No hope that the machine gives
any hint.
%
%
%%%%%%%%%%%%%%%%%%%%%%%%%%%%%%%%%%%%%%%%%%%%%%%%%%%%%%%%%%%%%%%%%%%
\subsection{The Google classification}
%%%%%%%%%%%%%%%%%%%%%%%%%%%%%%%%%%%%%%%%%%%%%%%%%%%%%%%%%%%%%%%%%%%
%
Though it does not use Kolmogorov complexity, we now present
another recent approach by Vitanyi and Cilibrasi \cite{cilibrasi07}
to classification which leads to a very performing tool.
%
%------------------------------------------------------------------
\subsubsection{The normalized Google distance $NGD$}\label{sss:ngd}
%------------------------------------------------------------------
This quite original method is based on the huge data bank
constituted by the world wide web and the Google search engine
which allows for basic queries using conjunction of keywords.
\\
Observe that the web is not a data base, merely a data bank,
since the data on the web are not structured as data of a data
base.
\medskip\\
Citing  \cite{evangelista06}, the idea of the method is as follows:
\begin{quote}
When the Google search engine is used to search for the word $x$,
Google dsiplays the number of hits that word $x$ has.
The ratio of this number to the total number of webpages indexed
by Google represents the probability that word $x$ appears on a webpage [...]
If word $y$ has a higher conditional probability to appear on a web
page, given that word $x$ also appears on the webpage, than it does
by itself, then it can be concluded that words $x$ and $y$ are
related.
\end{quote}
Let's cite an example from Cilibrasi and Vitany \cite{ciliercim05}
which we complete and update the figures.
The searches for the index term``horse", ``rider" and ``molecule"
respectively return $156$, $62.2$ and $45.6$ million hits.
Searches for pairs of words ``horse rider" and ``horse molecule"
respectively return $2.66$ and $1.52$ million hits.
These figures stress a stronger relation between the words
``horse" and ``rider" than between ``horse" and ``molecule".
\\
Another example with famous paintings:
``Dejeuner sur l'herbe",``Moulin de la Galette" and ``la Joconde".
Let refer them by a, b, c.
Google searches for a, b, c respectively give
$446~000$, $278~000$ and $1~310~000$ hits.
As both the searches for a+b, a+c and b+c, they respectively give
$13~700$, $888$ and $603$ hits. Clearly, the two paintings by Renoir
are more often cited together than each one is with the painting by
da Vinci.
\\
In this way, the method regroups paintings by artists, using what
is said about these paintings on the web.
But this does not associate the painters to groups of paintings.
\medskip\\
Formally, Cilibrasi and Vitany \cite{ciliercim05,cilibrasi07}
define the normalized Google distance as follows:
$$
NGD(x,y) =\frac{\max(\log f(x),\log f(y))-\log f(x,y)}
{\log M -\min(\log f(x),\log f(y))}
$$
where $f(z_1,...)$ is the number of hits for the conjunctive
query $z_1,...$ and $M$ is the total number of webpages that
Google indexes.
%
%------------------------------------------------------------------
\subsubsection{Discussing the method}
%------------------------------------------------------------------
1. The number of objects in a future classification and that of
canonical representatives of the different corpus is not chosen
in advance nor even boundable in advance and it is constantly
moving.
This dynamical and uncontrolled feature is a totally new
experience.
\medskip\\
2. Domains a priori completely rebel to classification as is
the pictorial domain (no normalization of paintings is possible)
can now be considered.
Because we are no more dealing with the paintings themselves
but with what is said about them on the web.
And, whereas the ``pictorial language" is merely a metaphor,
this is a true ``language" which deals with keywords and their
relations in the texts written by web users.
\medskip\\
3. However, there is a big limitation to the method, that of
a closed world: {\em the World according to Google},
{\em Information according to Google}...
\\ If Google finds something, one can check its pertinence.
Else, what does it mean? Sole certainty, that of uncertainty.
\\
When failing to get hits with several keywords, we give up the
original query and modify it up to the point Google gives some
pertinent answers.
\\ So that failure is as negation in Prolog which is much weaker
than classical negation.
It's reasonable to give up a query and accordingly consider
the related conjunction as meaningless. However, one should keep
in mind that this is relative to the close
- and relatively small - world of data on the web,
the sole world accessible to Google.
\\
When succeeding with a query, the risk is to stop on this
succeeding query and
\\- forget that previous queries have been tried which failed,
\\- omit going on with some other queries which could possibly
lead to more pertinent answers.
\medskip\\
There is a need to formalize information on the web and the
relations ruling the data it contains. And also the notion
of pertinence. A mathematical framework is badly needed.
\\
This remarkable innovative approach is still in its infancy.
%
%
%%%%%%%%%%%%%%%%%%%%%%%%%%%%%%%%%%%%%%%%%%%%%%%%%%%%%%%%%%%%%%%%%%%
\subsection{Some final remarks}
%%%%%%%%%%%%%%%%%%%%%%%%%%%%%%%%%%%%%%%%%%%%%%%%%%%%%%%%%%%%%%%%%%%
%
These approaches to classification via compression
and Google search of the web are really provocative.
They allow for classification of diverse corpus along a
top-down operational mode as opposed to bottom-up grouping.
\\
Top-down since there is no prerequisite of any a priori knowledge
of the content of the texts under consideration.
One gets information on the texts without entering their
semantics, simply by compressing them or counting hits with Google.
This has much resemblance with statistical methods which point
correlations to group objects. Indeed, compressors and Google
use a large amount of statistical expertise.
\\
On the opposite, a botton-up approach uses keywords which have to
be previously known so that we already have in mind what the
groups of the classification should be.
\medskip\\
Let's illustrate this top-down versus bottom-up opposition
by contrasting three approaches related to the classical
comprehension schema.
\medskip\\
{\em Mathematical approach.}\\
This is a global, intrinsically deterministic approach
along a fundamental dichotomy: true/false,
provable/inconsistent.
A quest for absoluteness based on certainty.
This is reflected in the classical comprehension schema
$$
\forall y\ \exists Z\ \ Z=\{x\in y \mid \+P(x)\}
$$
where $\+P$ is a property fixed in advance.
\medskip\\
{\em Probabilistic approach.}\\
In this pragmatic approach uncertainty is taken into consideration,
it is bounded and treated mathematically.
This can be related to a probabilistic version
of the comprehension schema where the truth of $\+P(x)$ is
replaced by some limitation of the uncertainty: the probability
that $x$ satisfies $\+P$ is true is in a given interval.
Which asks for a two arguments property $\+P$ :
$$
\forall y\ \exists Z\ \ Z=\{x\in y \mid
\mu(\{\omega\in\Omega\mid\+P(x,\omega)\})\in I\}
$$
where $\mu$ is a probability on some space $\Omega$
and $I$ is some interval of $[0,1]$.
\medskip\\
The above mathematical and probabilistic approaches are
bottom-up. One starts with a given $\+P$ to group objects.
\medskip\\
{\em Google approach.}\\
Now, there is no idea of the interval of uncertainty.
Google may give 0\% up to 100\% of pertinent answers.
It seems to be much harder to put in a mathematical framework.
But this is quite an exciting approach, one of the few
top-down ones together with the compression approach and
those based on statistical inference.
This Google approach reveals properties, regularity laws.
%
%
%%%%%%%%%%%%%%%%%%%%%%%%%%%%%%%%%%%%%%%%%%%%%%%%%%%%%%%%%%%%%%%%%%%
%%%%%%%%%%%%%%%%%%%%%%%%%%%%%%%%%%%%%%%%%%%%%%%%%%%%%%%%%%%%%%%%%%%
%%%%%%%%%%%%%%%%%%%%%%%%%%%%%%%%%%%%%%%%%%%%%%%%%%%%%%%%%%%%%%%%%%%
%%%%%%%%%%%%%%%%%%%%%%%%%%%%%%%%%%%%%%%%%%%%%%%%%%%%%%%%%%%%%%%%%%%
%%%%%%%%%%%%%%%%%%%%%%%%%%%%%%%%%%%%%%%%%%%%%%%%%%%%%%%%%%%%%%%%%%%
%%%%%%%%%%%%%%%%%%%%%%%%%%%%%%%%%%%%%%%%%%%%%%%%%%%%%%%%%%%%%%%%%%%
%%%%%%%%%%%%%%%%%%%%%%%%%%%%%%%%%%%%%%%%%%%%%%%%%%%%%%%%%%%%%%%%%%%
%%%%%%%%%%%%%%%%%%%%%%%%%%%%%%%%%%%%%%%%%%%%%%%%%%%%%%%%%%%%%%%%%%%
%%%%%%%%%%%%%%%%%%%%%%%%%%%%%%%%%%%%%%%%%%%%%%%%%%%%%%%%%%%%%%%%%%%
%%%%%%%%%%%%%%%%%%%%%%%%%%%%%%%%%%%%%%%%%%%%%%%%%%%%%%%%%%%%%%%%%%%
%

\end{document}